\documentclass[10pt]{amsart}

\usepackage[headings]{fullpage}

\usepackage{latexsym,amssymb}

\usepackage{amsfonts}

\usepackage{amsmath,amsthm,amssymb}

\author[W. Jeon, I. Kapovich, C. Leininger and K. Ohshika] {Woojin Jeon, Ilya Kapovich, Christopher Leininger, and Ken'ichi Ohshika}

\address{\tt School of Mathematics,
 KIAS, Hoegiro 87, Dongdaemun-gu,
     Seoul, 130-722, Korea }
    \email{\tt jwoojin\char`\@ kias.re.kr}

\address{\tt Department of Mathematics, University of Illinois at
Urbana-Champaign, 1409 West Green Street, Urbana, IL 61801, \indent USA
\newline \indent http://www.math.uiuc.edu/\~{}kapovich/} \email{\tt
kapovich@math.uiuc.edu}

\address{\tt  Department of Mathematics, University of Illinois at Urbana-Champaign,
  1409 West Green Street, Urbana, IL 61801, \indent USA
  \newline \indent http://www.math.uiuc.edu/\~{}clein/} \email{\tt clein@math.uiuc.edu}
  
\address{\tt Department of Mathematics, Graduate School of Science, Osaka University, Toyonaka, Osaka 560-0043, Japan}\email{\tt ohshika@math.sci.osaka-u.ac.jp}

\title[Conical limit points and the Cannon-Thurston map]{Conical limit points and the Cannon-Thurston map }

\newtheorem{theor}{Theorem}

\newtheorem{thm}{Theorem}[section] \newtheorem{lem}[thm]{Lemma}
\newtheorem{cor}[thm]{Corollary} 
\newtheorem{prop}[thm]{Proposition} \theoremstyle{definition}
\newtheorem{defn}[thm]{Definition}

\newtheorem{conv}[thm]{Convention} 
\newtheorem{exmp}[thm]{Example}

\newtheorem*{theoremA}{Theorem~\ref{thm:A}}
\newtheorem*{theoremB}{Theorem~\ref{thm:B}}
\newtheorem*{theoremC}{Theorem~\ref{thm:C}}
\newtheorem*{theoremD}{Theorem~\ref{thm:D}}

\def\strutdepth{\dp\strutbox}
\def \ss{\strut\vadjust{\kern-\strutdepth \sss}}
\def \sss{\vtop to \strutdepth{
\baselineskip\strutdepth\vss\llap{$\diamondsuit\;\;$}\null}}

\def\strutdepth{\dp\strutbox}
\def \sst{\strut\vadjust{\kern-\strutdepth \ssss}}
\def \ssss{\vtop to \strutdepth{
\baselineskip\strutdepth\vss\llap{$\spadesuit\;\;$}\null}}

\def\strutdepth{\dp\strutbox}
\def \ssh{\strut\vadjust{\kern-\strutdepth \sssh}}
\def \sssh{\vtop to \strutdepth{
\baselineskip\strutdepth\vss\llap{$\heartsuit\;\;$}\null}}

\newcommand{\Sph}{\mathbb S}
\newcommand{\Z}{\mathbb Z}

\newcommand{\Hy}{\mathbb H}

\def\epsilon{\varepsilon}
\def\phi{\varphi}

\def\hat{\widehat}

\newcommand{\Out}{\mbox{Out}}
\newcommand{\Aut}{\mbox{Aut}}

\newcommand{\gnn}{(g_n)_{n \ge 1}}
\newcommand{\gn}{(g_n)}
\newcommand{\gnkk}{(g_{n_k})_{k \geq 1}}
\newcommand{\gnk}{(g_{n_k})}
\newcommand{\limk}{\displaystyle{\lim_{k \to \infty}}}
\newcommand{\limn}{\displaystyle{\lim_{n \to \infty}}}
\newcommand{\PSL}{\rm PSL}

\begin{document}

\begin{abstract}
Let $G$ be a non-elementary word-hyperbolic group acting as a convergence group on a compact metrizable space $Z$ so that there exists a continuous $G$-equivariant map $i:\partial G\to Z$, which we call a \emph{Cannon-Thurston map}. We obtain two characterzations (a dynamical one and a geometric one) of conical limit points in $Z$ in terms of their pre-images under the Cannon-Thurston map $i$. As an application we prove, under the extra assumption that the action of $G$ on $Z$ has no accidental parabolics, that if the map $i$ is not injective then there exists a non-conical limit point $z\in Z$ with $|i^{-1}(z)|=1$. This result applies to most natural contexts where the Cannon-Thurston map is known to exist, including subgroups of word-hyperbolic groups and Kleinian representations of surface groups. As another application, we prove that if $G$ is a non-elementary torsion-free word-hyperbolic group then there exists $x\in \partial G$ such that $x$ is not a \lq\lq controlled concentration point" for the action of $G$ on $\partial G$.
\end{abstract}



\maketitle

\section{Introduction}

Let $G$ be a Kleinian group, that is, a discrete subgroup of the isometry group of hyperbolic space $G \le  {\rm Isom}^+(\Hy^n)$.  In \cite{BeMa}, Beardon and Maskit defined the notion of a {\em conical limit point} (also called a {\em point of approximation} or {\em radial limit point}), and used this to provide an alternative characterization of geometric finiteness.  Gehring and Martin~\cite{GeMa} abstracted the notion of Kleinian group to that of a {\em convergence group} acting on $\Sph^{n-1}$, which was then further generalized, for example by Tukia in \cite{Tu98}, to actions on more general compact metric spaces (see Definition~\ref{Def:convergence}).  This generalization includes, for example, the action of a discrete group of isometries of a proper, Gromov--hyperbolic, metric space on its boundary at infinity; see \cite{Fre,Tu98}.  Conical limit points can be defined in this level of generality (see Definition~\ref{Def:conical}), and play a key role in the convergence group characterization of word-hyperbolic groups by Bowditch~\cite{Bow98}, of relatively hyperbolic groups by Yaman~\cite{Y}, and of quasi-convex subgroups of word-hyperbolic groups by Swenson~\cite{Swe}, and arise in numerous other results in topology, geometry, and dynamics; see, for example, \cite{CaJu,Gab92,ABT,Ger,Fen,KeLe1,KeLe2} for other results involving convergence groups and conical limit points.

In their 1984 preprint, published in 2007~\cite{CT}, Cannon and Thurston proved the following remarkable result.  If $M$ is a closed hyperbolic 3-manifold fibering over a circle with fiber a closed surface $\Sigma$, then the inclusion $\Hy^2=\widetilde\Sigma\subset \widetilde M=\Hy^3$ extends to a continuous surjective map $\Sph^1=\partial \Hy^2\to \partial \Hy^3=\Sph^2$, equivariant with respect to $\pi_1(\Sigma)$ which is acting as a convergence group on both.  Based on this we make the following general, abstract definition.


\begin{defn}\label{D:CT}
When $G$ is a word-hyperbolic group acting as a convergence group on a compact metrizable space $Z$, a map $i:\partial G\to Z$ is called a \emph{Cannon-Thurston map} if $i$ is continuous and $G$-equivariant. 
\end{defn}

Under some mild assumptions, it is known that if a Cannon-Thurston map $i:\partial G\to Z$ exists, then it is unique; see Proposition~\ref{prop:ct-unique} below.  Of particular interest is the case that a non-elementary word-hyperbolic group $G$ acts on a proper, Gromov hyperbolic, geodesic metric space $Y$, properly discontinuously by isometries, and without accidental parabolics (see Definition \ref{def:ap}).  In this case, if there exists a Cannon--Thurston map $i \colon \partial G \to \partial Y$, then it is known to be unique and to extend to a $G$--equivariant continuous map $G \cup \partial G \to Y \cup \partial Y$ (see Proposition~\ref{prop:ct-geom}).  A special subcase of interest is when $G_1,G_2$ are non-elementary word-hyperbolic groups, with $G_1 \le G_2$ acting on (the Cayley graph of) $G_2$ by restriction of the left action of $G_2$ on itself.  Here a {\em Cannon--Thurston map} is classically defined as a continuous extension $G_1 \cup \partial G_1 \to G_2 \cup \partial G_2$ of the inclusion of $G_1 \to G_2$.  By Proposition~\ref{prop:ct-geom}, the existence of such a map is equivalent to the existence of a Cannon--Thurston map in the sense of Definition~\ref{D:CT} for the induced action of $G_1$ on $\partial G_2$.  Quasi-isometrically embedded subgroups $G_1 \le G_2$ of word-hyperbolic groups provide examples where Cannon--Thurston maps exist.  However, Cannon--Thurston's original result~\cite{CT} described above implies that for the word-hyperbolic groups $G_1 = \pi_1(\Sigma) \le \pi_1(M)=G_2$, there is a Cannon--Thurston map $\partial G_1 \to \partial G_2$, but here $G_1$ is exponentially distorted in $G_2$.  Subsequent work of Mitra~\cite{M2,M4} showed that there are many other interesting situations where $G_1$ is not quasiconvex in $G_2$ but where the Cannon-Thurston map nevertheless does exist (see also \cite{BRi}).   On the other hand, a recent remarkable result of Baker and Riley~\cite{BR} proves that there exists a word-hyperbolic group $G_2$ and a word-hyperbolic (in fact, nonabelian free) subgroup $G_1\le G_2$ such that the Cannon-Thurston map $i: \partial G_1\to \partial G_2$ does not exist.

Generalizing the Cannon--Thurston example from \cite{CT} in another direction, one can consider other actions of $G = \pi_1(\Sigma)$, the fundamental group of a closed, orientable surface of genus at least $2$, acting properly discontinuously by isometries on $\Hy^3$, i.e. as a classical Kleinian surface group.  The first partial results beyond those in \cite{CT} about the existence of Cannon--Thurston maps for such actions of $G$ on $\Hy^3$ are due to Minsky \cite{Min94}.   
Extending beyond the case $G = \pi_1(\Sigma)$, there have been numerous results on the existence of Cannon--Thurston maps of various types (not necessarily fitting into Definition~\ref{D:CT}), especially for Kleinian groups \cite{F80, Kl99, Mc01, Sou, Miy, Bow07, M09, M0910, M10, M10b, Bow02}.
Recently, Mj~\cite{M14} has shown that for {\em any} properly discontinuous action on $\Hy^3$ without accidental parabolics, there exists a Cannon--Thurston map, using the theory of model manifolds which were developed by Minsky. 
There are extensions of the Cannon-Thurston maps also for subgroups of mapping class groups \cite{LMS},  and in other related contexts~\cite{Ger12,GP13}.

Mj has also shown~\cite{M07} that in the case of classical Kleinian surface groups without parabolics, the non-injective points of a Cannon-Thurston map are exactly the endpoints of the lifts of the ending laminations to the domains of discontinuity. This characterization of non-injective points of Cannon-Thurston maps has some applications: for instance, the first and the fourth authors have used this to prove {\it the measurable rigidity} for Kleinian groups (see \cite{JO}), which is a generalization of the results by Sullivan \cite{Sul82} and Tukia \cite{Tu89}.
Also, using the same kind of characterization for free classical Kleinian groups, Jeon-Kim-Ohshika-Lecuire \cite{JKOL} gave a criterion for points on the boundary of the Schottky space to be primitive stable.

Another reason to be interested in understanding injective points of Cannon-Thurston maps comes from the study of dynamics and geometry of fully irreducible elements of $\Out(F_N)$.  If $\phi\in \Out(F_N)$ is an atoroidal fully irreducible element then the mapping torus group $G_\phi=F_N\rtimes_\phi \mathbb Z$ is word-hyperbolic and the Cannon-Thurston map $i:\partial F_N\to \partial G_\phi$ exists by the result of \cite{M2}. In this case, if $T_\pm$ are the "attracting" and "repelling" $\mathbb R$-trees for $\phi$, there are associated $\mathcal Q$-maps (defined in \cite{CHL2}) $\mathcal Q_+: \partial F_N\to \hat T_+=\overline T_+\cup \partial T_+$ and $\mathcal Q_-: \partial F_N\to \hat T_-\overline T_-\cup \partial T_-$ (here $\overline{T_\pm}$ denotes the metric completion of $\overline{T_\pm}$). These maps play an important role in the index theory of free group automorphisms, particularly for the notion of $\mathcal Q$-index; see \cite{CHL,CHL2,CH12,CH14}. It is shown in \cite{KL6} that a point $x\in \partial F_N$ is injective for the Cannon-Thurston map $i$ if and only if $x$ is injective for both $\mathcal Q_+$ and $\mathcal Q_-$.

There are a number of results in the literature which prove in various
situations where the Cannon-Thurston map exists that every conical
limit point is ``injective", that is, has exactly one pre-image under
the Cannon-Thurston map; see, for example, \cite{K95,LLR,Ger12}.  We discuss some of these facts in more detail after the statement of Theorem~\ref{thm:B} below.
These results naturally raise the question
whether the converse holds, that is whether a point
with exactly one pre-image under the Cannon-Thurston map must be a
conical limit point. (The only result in the literature dealing with
this converse direction is Theorem~8.6 in \cite{K95}, which
incorrectly claims that every ``injective'' limit point is conical in
the original setting of a closed hyperbolic 3-manifold fibering over a
circle.)  We show in Theorem~\ref{thm:C} below that the converse statement fails in great generality and prove that, under rather mild assumptions, if a Cannon-Thurston
map exists and is not injective then there \emph{always} exists a
non-conical limit point with exactly one pre-image under the
Cannon-Thurston map.




In this paper, given a non-elementary convergence action of a word hyperbolic group $G$ on a compact metrizable space $Z$, such that the Cannon-Thurston map $i:\partial G\to Z$ exists, we give two characterizations  (a dynamical one and a geometric one)  of conical limit points $z\in Z$ in terms of their pre-images under the map $i$. 

To state these characterizations we need to introduce some definitions. Under the above assumptions, denote $L_i=\{(x,y)| x,y\in \partial G, i(x)=i(y), \text{ and } x\ne y\}$. 
We say that a point $x \in \partial G$ is {\em asymptotic to $L_i$} if for {\em every} conical sequence $\{g_n\}_{n=1}^\infty$ for $x$ with pole pair $(x_-,x_+)$, we have $(x_-,x_+)\in L_i$,  that is, $i(x_-) = i(x_+)$.   (See Definition~\ref{Def:conical} below for the notions of a conical sequence and pole pair).

The following result provides a dynamical characterization of conical
limit points in $Z$:

\begin{theor}\label{thm:A} Suppose $G$ is word-hyperbolic and acts on the compact, metrizable space $Z$ as a non-elementary convergence group, and suppose $i \colon \partial G \to Z$ is a Cannon-Thurston map.  Let $z\in i(\partial G)$. Then:
\begin{enumerate}
\item The point $z \in Z$ is not a conical limit point for the action of $G$ on $Z$  if and only if some point $x \in i^{-1}(z)$ is asymptotic to $L_i$.  

\item If $|i^{-1}(z)| > 1$, then any $x \in i^{-1}(z)$ is asymptotic to $L_i$, and hence $z$ is non-conical.
\end{enumerate}
\end{theor}

We also provide a geometric counterpart of Theorem~\ref{thm:A}:

\begin{theor}\label{thm:B} Let $G$ be a word-hyperbolic group and let $Z$ be a compact metrizable space equipped with a non-elementary convergence action of $G$ such that the Cannon-Thurston map $i: \partial G\to  Z$ exists and such that $i$ is not injective. Let $X$ be a $\delta$-hyperbolic (where $\delta\ge 1$) proper geodesic metric space equipped with a properly discontinuous cocompact isometric action of $G$ (so that $\partial G$ is naturally identified with $\partial X$).

Let $x\in \partial G$, let $z=i(x)\in Z$ and let $\rho$ be a geodesic ray in $X$ limiting to $x$.

Then the following are equivalent:
\begin{enumerate}
\item The point $z$ is a  conical limit point for the action of $G$ on $Z$.
\item There exist a geodesic segment $\tau=[a,b]$ in $X$ of length $\ge 100\delta$ and an infinite sequence of distinct elements $g_n\in G$ such that the $20\delta$-truncation $\tau'$ of $\tau$ is not a coarse $X$-leaf segment of $L_i$ and such that 
for each $n\ge 1$ the segment $g_n\tau$ is contained in a $6\delta$-neighborhood of $\rho$, and that $\lim_{n\to\infty} g_na=\lim_{n\to\infty} g_nb=x$.
\end{enumerate}
\end{theor}

See the definition of a ``coarse leaf segment"  and of other relevant terms in Section~\ref{sect:geom}.

Theorem~\ref{thm:A} is used in the proof of Theorem~\ref{thm:C},  as discussed in more details below. Theorem~\ref{thm:B} is used as a key ingredient in the proof of Theorem~6.6 of \cite{DKT}.

Theorems \ref{thm:A},\ref{thm:B} are partially motived by the result of  M.~Kapovich~\cite{K95}, who proved that in the setting of Cannon and Thurston's original construction \cite{CT} of a Cannon--Thurston map $i \colon \partial G \to \Sph^2$, for $G = \pi_1(\Sigma)$, if $z\in  \partial \Sph^2=\partial G$ has $|i^{-1}(z)| \ge 2$ then $z$ is not a conical limit point for the action of $G$ on $\Sph^2$.  This result was extended by Leininger, Long and Reid~\cite{LLR}, who proved that the same result for any doubly degenerate Kleinian representation (where $i$ exists from \cite{M14}), and later by Gerasimov~\cite{Ger12}, for arbitrary "$\times$-actions".   In fact, part (2) of Theorem~\ref{thm:A} follows from a general result of Gerasimov~\cite[Proposition 7.5.2]{Ger12} about conical limit points for $\times$-actions. Gerasimov also explained to us how one can derive part (1) of Theorem~\ref{thm:A}  from the results of \cite{Ger12} using a result of Bowditch. We provide a short direct proof of Theorem~\ref{thm:A} here.

It is known by results of Swenson~\cite{Swe} and Mitra~\cite{M11} that in the geometric context, where the Cannon-Thurston map $i:\partial G\to \partial Y$ arises from a properly discontinuous isometric action of a word-hyperbolic $G$ on a proper Gromov-hyperbolic space $Y$,  if $i:\partial G\to \partial Y$  is injective then the orbit-map $G\to Y$ is a quasi-isometric embedding; see Proposition~\ref{prop:inj} below for a precise statement. In this case every limit point $z\in \partial Y$ is conical and has exactly one pre-image under $i$.  Therefore Theorem~\ref{thm:A}  implies:

\begin{cor}\label{cor:noninj}
Let $G$ be a non-elementary word-hyperbolic group equipped with a properly discontinuous isometric action on a proper geodesic Gromov-hyperbolic space $Y$  without accidental parabolics. Suppose that the Cannon-Thurston map $i:\partial G\to\partial Y$ exists.

Then there exists $z\in i(\partial G)$ such that $z$ is a non-conical limit point for the action of $G$ on $\partial Y$ if and only if $i$ is not injective.
\end{cor}
\begin{proof}
If $i$ is not injective and $x_1,x_2\in \partial G$ are points such that $x_1\ne x_2$ and that $i(x_1)=i(x_2)$ then, by part (2) of Theorem~\ref{thm:A}, $z=i(x_1)=i(x_2)$ is not conical.
If the map $i$ is injective, then, since $\partial G$ and $i(\partial G)\subseteq \partial Y$ are compact and Hausdorff, the map $i$ is a $G$-equivariant homeomorphism between $\partial G$ and $i(\partial G)$.  Since $G$ is hyperbolic, every point of $\partial G$ is conical for the action of $G$ on $\partial G$, see~\cite{Tu98}.  Therefore every $z\in i(\partial G)$ is conical for the action of $G$ on $i(\partial G)$ and hence, by Lemma~\ref{lem:subset}, also for the action of $G$ on $\partial Y$.
\end{proof}

The main result of this paper is:

\begin{theor} \label{thm:C}
Suppose a word-hyperbolic group $G$ acts on a compact metrizable space $Z$ as a non-elementary convergence group without accidental parabolics, and suppose that there exists a non-injective Cannon--Thurston map $i \colon \partial G \to Z$.  Then there exists a non-conical limit point $z \in Z$ with $|i^{-1}(z)| = 1$.
\end{theor}

Theorem~\ref{thm:C} applies whenever $G_1$ is a non-elementary non-quasiconvex word-hyperbolic subgroup of a word-hyperbolic group $G_2$ such that the Cannon-Thurston map $\partial G_1\to\partial G_2$ exists. Similarly, Theorem~\ref{thm:C} applies whenever $\Sigma$ is a closed hyperbolic surface and $\pi_1(\Sigma)$ is equipped with a properly discontinuous isometric action on $\Hy^3$ without accidental parabolics, assuming that the Cannon-Thurston map $\mathbb S^1=\partial \pi_1(\Sigma)\to \partial \Hy^3=\mathbb S^2$ exists  and is non-injective.

Brian Bowditch (private communication) showed us another argument for obtaining the conclusion of Theorem~\ref{thm:C} for a large class of Kleinian groups, including the original case of a closed hyperbolic 3-manifold fibering over a circle. His argument is different from the proof presented in this paper and relies on the Kleinian groups and 3-manifold methods.

Another application of our results concerns \lq\lq controlled concentration points".
Originally, the notion of a controlled concentration point was defined for a properly discontinuous isometric action of a torsion-free group $G$ on  $\Hy^n$.
A point $x$ in $\partial \Hy^n$ is called a controlled concentration point of $G$ when $x$ has a neighborhood $V$ such that for any neighborhood $U$ of $x$ there is $g \in G$ with $g U \subset V$ and $x \in g(V)$.
This is equivalent to saying that there is a sequence of elements $(g_n)_{n \ge 1} \subset G$ such that $g_n(x) \rightarrow x$ and $(g_n|_{\partial \Hy^n \setminus \{x\}})$ converges locally uniformly to a constant map to some point $y \neq x$.
Aebischer, Hong and McCullough \cite{AHM} showed that a limit point $x\in \partial \Hy^n$ is a controlled concentration point if and only if it is an endpoint of a lift of a recurrent geodesic ray in $M:=\Hy^n/G$. A geodesic ray $\alpha(t)$ in $M$ is called {\it recurrent} if for any $t_0$, there exists a sequence $\{t_i\}$ with $t_i\rightarrow\infty$ such that $\displaystyle{\lim_{i \to \infty} \alpha'(t_i)} =  \alpha'(t_0)$ in the unit tangent bundle of $M$. 
They also showed there exist non-controlled concentration points in the limit set of a rank-2 Schottky group.

We generalize the notion of controlled concentration points to points at infinity of general word-hyperbolic groups by adopting the latter condition above as its definition; see Definition~\ref{defn:controlled} below. As an application of  Theorem C, we get the following existence theorem of non-controlled concentration points:

\begin{theor}\label{thm:D}
Let $G$ be a non-elementary torsion-free word-hyperbolic group. Then there exists $x\in \partial G$ which is not a controlled concentration point.
\end{theor}

In Appendix~\ref{s:appendix} we discuss several specific situations where  where the Cannon-Thurston map $i:\partial G\to Z$ is known to exist and where a more detailed description of the lamination $L_i$ is known.

\noindent{\bf Acknowledgements:}  We thank Victor Gerasimov and Leonid Potyagailo for their interest in the results of this paper and for bringing to our attention the paper~\cite{Ger12} and explaining to us the relationship between the results of \cite{Ger12} and our Theorem~\ref{thm:A}.  We are  grateful to  Brian Bowditch for helpful discussions regarding conical limit points and the Cannon-Thurston map, and to Sergio Fenley, Darren Long, and Alan Reid for useful conversations about convergence actions and conical limit points.  We thank the referee for useful comments.

\noindent{\bf Funding:} The second author was partially supported by a
Collaboration Grant no. 279836 from the Simons Foundation and by the
NSF grant DMS-1405146. The third
  author was partially supported by the NSF grant DMS-1207183 and DMS-1510034.  The last author was partially supported by the JSPS Grants-in-Aid 70183225.

\section{Definitions and basic facts}

\subsection{Convergence groups}

\begin{defn}[Convergence action] \label{Def:convergence}
An action of a group $G$ on a compact metrizable space $Z$ by homeomorphisms is called a \emph{convergence action} (in which case we also say that $G$ acts on $Z$ as a \emph{convergence group}) if for any infinite sequence $\gnn$ of distinct elements of $G$ there exist $a,b\in Z$ and a subsequence $\gnkk$ of $\gn$, called a {\em convergence subsequence}, such that the sequence of maps $\{g_{n_k}|_{Z\setminus \{a\}} \}$ converges uniformly on compact subsets to the constant map $c_b \colon Z\setminus \{a\}\to Z$ sending $Z\setminus \{a\}$ to $b$.   In this case we call $(a,b,\gnk)$ the \emph{convergence subsequence data}.  The action is called \emph{elementary} if either $G$ is finite or $G$ preserves a subset of $Z$ of cardinality $\le 2$, and it is called \emph{non-elementary} otherwise.

\end{defn}
Note that if $G$ acts as a convergence group on $Z$ and if $Z'\subseteq Z$ is a nonempty closed $G$-invariant subset, then the restricted action of $G$ on $Z'$ is also a convergence action.

For a group $G$ acting of a set $Z$ and for $g\in G$ denote $Fix_Z(g):=\{z\in Z| gz=z\}$. The following is a basic fact about convergence groups, see \cite[Lemma 3.1]{Bow99} and \cite{Tu98}:
\begin{prop}
Suppose $G$ acts as a convergence group on a compact metrizable space $Z$ and let $g\in G$. Then exactly one of the following occurs:
\begin{enumerate}
\item The element $g$ has finite order in $G$; in this case $g$ is said to be \emph{elliptic}.
\item The element $g$ has infinite order in $G$ and the fixed set $Fix_Z(g)$ consists of a single point; in this case $g$ is called \emph{parabolic}.
\item The element $g$ has infinite order in $G$ and the fixed set $Fix_Z(g)$ consists of two distinct points; in this case $g$ is called \emph{loxodromic}.
\end{enumerate}
Moreover, for every $k\ne 0$ the elements $g$ and $g^k$ have the same type;  also in cases (2) and (3) we have  $Fix_Z(g)=Fix_Z(g^k)$ and  the group $\langle g\rangle$ acts properly discontinuously on $Z\setminus Fix_Z(g)$. Additionally, if $g\in G$ is loxodromic then $\langle g\rangle$ acts properly discontinuously and cocompactly.
\end{prop}
It is also known that if $g\in G$ is parabolic with a fixed point $a\in Z$ then for every $z\in  Z$ we have $\limn g^nz=\displaystyle{\lim_{n\to-\infty}} g^nz=a$. Also, if $g\in G$ is loxodromic then we can write $Fix_Z(g)=\{a_-,a_+\}$ and for every $z\in Z\setminus \{a_+\}$ we have $\limn g^nz=a_-$, and  for every $z\in Z\setminus \{a_-\}$ we have $\displaystyle{\lim_{n\to-\infty}} g^nz=a_+$, and these convergences are uniform on compact subsets of $Z\setminus \{a_-,a_+\}$.

\begin{defn}[Limit set] If $G$ acts on $Z$ as a non-elementary convergence group, there exists a unique minimal nonempty closed $G$-invariant subset $\Lambda(G)\subseteq Z$ called the \emph{limit set} of $G$ in $Z$.
In this case $\Lambda(G)$ is perfect and hence $\Lambda(G)$ is infinite~\cite{Tu98}. If $\Lambda(G) = Z$, then we say that the action of $G$ on $Z$ is \emph{minimal}. 
\end{defn}

\begin{defn}[Conical limit point] \label{Def:conical}
Let $G$ act on $Z$ as a convergence group. A point $z\in Z$ is called a \emph{conical limit point} for the action of $G$ on $Z$ if there exist an infinite sequence $\gnn$ of distinct elements of $G$ and a pair of distinct points $z_-,z_+\in Z$ such that $\limn g_n z = z_+$ and that $(g_n|_{Z\setminus \{z\}} )$ converges uniformly on compact subsets to the constant map $c_{z_-} \colon Z\setminus \{z\}\to Z$ sending $Z\setminus \{z\}$ to $z_-$. We call such a sequence $g_n$ a
\emph{conical sequence} for $z$, and the pair $(z_-,z_+)$  the  \emph{pole pair} corresponding to $z$ and $(g_n)_{n\ge 1}$.  If every point of $Z$ is a conical limit point, then the action is called a {\em uniform convergence action}.  In particular, in this case the action is minimal.
\end{defn}
Note that if $g\in G$ is loxodromic with $Fix_Z(g)=\{a_-,a_+\}$ then both $a_+,a_-$ are conical limit points for the action of $G$ on $Z$, and one can use $(g^n)_{n \ge 1}$ as the conical sequence with pole pair $(a_-,a_+)$.
 
As usual, for $\delta\ge 0$, a \emph{$\delta$-hyperbolic space} is a geodesic metric space $X$ such that for every geodesic triangle in $X$ each side of this triangle is contained in the $\delta$-neighborhood of the union of two other sides. A metric space $X$ is  \emph{Gromov-hyperbolic} if there exists $\delta\ge 0$ such that $X$ is $\delta$-hyperbolic. A finitely generated group $G$ is called \emph{word-hyperbolic} if for some (equivalently, any) finite generating set $\mathcal S$ of $G$, the Cayley graph of $G$ with respect to $\mathcal S$ is Gromov-hyperbolic.
See \cite{BridHae} for basic background information about Gromov-hyperbolic spaces and word-hyperbolic groups; also see \cite{KB} on the background regarding boundaries of hyperbolic spaces and of word-hyperbolic groups.

\begin{exmp}
Let $G$ be an infinite word-hyperbolic group, and write $\partial G$ to denote the Gromov boundary.  Then the action of $G$ on $\partial G$ is a uniform convergence action.  In fact, according to a result of Bowditch \cite{Bow98}, if the action of a group $G$ on a compact metrizable space $Z$ is a uniform convergence action, then $G$ is a word-hyperbolic and there is a $G$--equivariant homeomorphism between $\partial G$ and $Z$.  
\end{exmp}


\begin{lem}\label{lem:subset}
Let a group $G$ act as a convergence group on $Z$ and let $Z'\subseteq Z$ be a nonempty infinite closed $G$-invariant subset, and let $z\in Z'$.
Then $z$ is a conical limit point for the action of $G$ on $Z$ if and only if $z$ is a conical limit point for the action of $G$ on $Z'$. 
\end{lem} 
\begin{proof}
The ``only if" direction is obvious from the definition of a conical limit point.
Thus suppose that $z\in Z'$ is a conical limit point for the action of $G$ on $Z'$.  Let $(g_n)_n$ be a conical sequence for $z$ for the action of $G$ on $Z'$ and let  $(z_-,z_+)$ be the corresponding pole pair (also for the action of $G$ on $Z'$). Since $G$ acts on $Z$ as a convergence group, there exists convergence subsequence data $(a,b, \gnk)$.  By assumption $Z'$ is infinite and hence for every $x\in Z'\setminus \{a,z\}$ we have $\limn g_nx =  z_-$ and $\limn g_nx =  b$, it follows that $z_-=b$.  

We claim that $a=z$. Indeed, suppose that $a\ne z$. 
Since $z\in  Z\setminus \{a\}$, it follows that $\limn g_nz=b=z_-$. On the other hand, by assumption about $(z_-,z_+)$ being the pole pair for $z$ and $(g_n)_n$, it follows that $\limn g_nz=z_+$. By definition $z_-\ne z_+$, which gives a contradiction. Thus indeed $a=z$. Hence $z$ is a conical limit point for the action of $G$ on $Z$, as claimed.
\end{proof}

The following basic fact is well-known; see, for example, \cite{Bow99}. 
\begin{prop}\label{prop:bow} 
Let $G$ be a word-hyperbolic group acting as a non-elementary convergence group on a compact metrizable space $Z$.

Then a $G$-limit point  $z\in Z$ is a conical limit point for the action of $G$ on $Z$ if and only if for every point $s\in Z$ such that $s\ne z$ there exists an infinite sequence $g_n\in G$ of distinct elements of $G$ and points $s_\infty, z_\infty\in Z$ such that $s_\infty\ne z_\infty$ and such that  $\limn g_n z=z_\infty$ and $\limn g_n s=s_\infty$.
\end{prop}

\begin{defn} [Accidental parabolic]
\label{def:ap}
Let $G$ be an infinite word-hyperbolic group acting as a non-elementary convergence group on a compact metrizable space $Z$.  An {\em accidental parabolic} for this action is an infinite order element $g \in G$ such that $g$ acts parabolically on $Z$.
\end{defn}

\subsection{Cannon-Thurston map}

\begin{defn}[Cannon--Thurston map]
Let $G$ be a word-hyperbolic group acting as a non-elementary convergence group on a compact metrizable space $Z$.
A map $i \colon \partial G\to Z$ is called a \emph{Cannon--Thurston map} if $i$ is continuous and $G$-equivariant.
\end{defn}

\begin{lem} \label{L:CT fixed points}
Let $G$ be a word-hyperbolic group acting as a non-elementary convergence group on a compact metrizable space $Z$ and suppose $i \colon \partial G \to Z$ is a Cannon--Thurston map.  Then:

\begin{enumerate} 
\item If $g$ acts as a loxodromic on $Z$, then the attracting and repelling fixed points in $\partial G$ of $g$, respectively, are sent by $i$ to the attracting and repelling fixed points of $g$ in $Z$, respectively.  

\item If $g$ is an accidental parabolic, then there is exactly one fixed point for $g$ on $Z$, which is the $i$--image of the two fixed points in $\partial G$.
\end{enumerate}
\end{lem}
\begin{proof}
Let $g\in G$ be an element of infinite order. Denote by $g^\infty$ and $g^{-\infty}$ the attracting and repelling points for $g$ in $\partial G$ respectively. Since $i$ is $G$-equivariant and the points $g^{\pm \infty}\in \partial G$ are fixed by $g$, it follows that $i(\{g^\infty, g^{-\infty}\})\subseteq Fix_Z(g)$.  If $g$ is parabolic and $Fix_Z(g)=\{a\}$ it follows that $i(g^{\pm\infty})=a$. 

Suppose now that $g$ acts on $Z$ loxodromically. Since by assumption $G$ acts on $Z$ and hence on $i(\partial G)$ as a non-elementary convergence group, the set $i(\partial G)$ is infinite. Hence there exists $x\in \partial G$ such that $i(x)\not\in Fix_Z(g)$   (and hence $x\not\in \{g^\infty, g^{-\infty}\}$).  If $g$ acts loxodromically on $Z$ with $Fix_Z(g)=\{a_+,a_-\}$  then $\limn g^n x=g^\infty$ and hence, by continuity and $g$-equivariance of $i$,  $\limn g^n i(x)=i(g^\infty)$. On the other hand, by definition  of a loxodromic element, since $i(x)\ne a_-$ we have  $\limn g^n i(x)=a_+$.  Thus $i(g^\infty)=a_+$. Replacing $g$ by $g^{-1}$ we get $i(g^{-\infty})=a_-$.  
\end{proof}

\begin{prop}[Cannon--Thurston map unique]\label{prop:ct-unique}
Let $G$ be a word-hyperbolic group acting as a non-elementary convegence group on a compact metrizable space $Z$, then any two Cannon--Thurston maps $i,j \colon \partial G \to Z$, if they exist, must be equal.
\end{prop}
\begin{proof}
Since $i,j$ are continuous, they are determined by what they do to a dense set of points.  The set of attracting endpoints of any infinite order element $g \in G$ and its conjugates $\{g^h\}_{h \in G}$ forms such a dense set.  By Lemma \ref{L:CT fixed points}, $i$ and $j$ must agree on this set, hence must be equal.
\end{proof}
In the situation where Proposition~\ref{prop:ct-unique} applies, if a Cannon-Thurston map $i:\partial G\to Z$ exists, we will refer to $i$ as \emph{the Cannon-Thurston map}. 

There is a particular geometric situation where the Cannon-Thurston map has a more natural geometric meaning:

\begin{prop}\label{prop:ct-geom}
Let $G$ be a non-elementary word-hyperbolic group equipped with a properly discontinuous (but not necessarily co-compact) isometric action of a proper Gromov-hyperbolic geodesic metric space $Y$, so that every element of infinite order acts as a loxodromic isometry of $Y$.  Then the following hold:

\begin{enumerate}

\item Then $\partial Y$ is compact and $G$ acts on $\partial Y$ as a convergence group without accidental parabolics. (Thus Proposition~\ref{prop:ct-unique} applies.) 

\item Suppose the Cannon-Thurston map $i:\partial G\to\partial Y$ exists. Then for every $p\in Y$ the map
\[
f: G\cup \partial G\to Y\cup \partial Y
\]
given by $f(g)=gp$ for $g\in G$, and $f(x)=i(x)$ for $x\in \partial G$, is continuous for the hyperbolic compactification topologies on $G\cup \partial G$ and  $Y\cup \partial Y$.
\end{enumerate}

\end{prop}
\begin{proof}
Part (1) is well-known and due to Tukia~\cite{Tu94}. 

For part (2), note that the topology on $G$ is discrete. Thus we only need to check continuity of $f$ at points of $\partial G$.  Since $i$ is assumed to be continuous, it suffices to establish the following:\\

\noindent{\bf Claim.}  If $x\in \partial G$ and $(g_n)_{n\geq1} \subset G$ is an infinite sequence of distinct elements of $G$ such that $\limn g_n= x$ in $G\cup\partial G$ then $\limn g_np= i(x)$ in $Y\cup \partial Y$ for every $p\in Y$.\\

Assume that $x$ and $(g_n)$ are as in the Claim, but that the sequence $g_np$ does not converge to $i(x)$ in $Y\cup\partial Y$. Since $G$ acts properly discontinuously on $Y$, it follows that, after replacing $g_n$ by a subsequence, we have 
$\lim_{n\to\infty} g_np= z$ for some  $z\in\partial Y$ such that $z\ne i(x)$. Then there exist a subsequence $g_{n_k}$ and points $a,b\in \partial G$ and $c,d\in \partial Y$ such that 
$(g_{n_k}|_{\partial G\setminus \{a\}})$ converges uniformly on compact sets to the constant map to $b$, and $g_{n_k}|_{\partial Y\setminus \{c\}}$ converges uniformly on compact sets to the constant map to $d$.
Moreover, the fact that $\limn g_n= x$ implies that $x=b$ and, similarly, the fact that $\limn g_np= z$ implies that $z=d$ (see \cite{Tu94}).
Since the set $i(\partial G)$ is infinite, we can find $y\in \partial G$ such that  $y\ne a$ and $i(y)\ne c$. Then, on one hand, we have  $\lim_{k\to\infty} g_{n_k} i(y)=d=z$.
On the other hand, $\limk g_{n_k} y=b=x$ and therefore, by continuity of $i$, we have  $\limk g_{n_k} i(y)=i(x)$. This contradicts  $z\ne i(x)$.
\end{proof}

A general result of Mj shows that in this situation injectivity of the Cannon-Thurston map is equivalent to the orbit map $X\to Y$ being a quasi-isometric embedding~\cite[Lemma~2.5]{M11}:
\begin{prop}\label{prop:inj}
Let $G$ and $Y$ be as in Proposition~\ref{prop:ct-geom}  and Let $p\in Y$. Then the following conditions are equivalent:
\begin{enumerate}
\item The Cannon-Thurston map $i:\partial G\to\partial Y$ exists and is injective.
\item The orbit map $G\to Y$, $g\mapsto gp$, is a quasi-isometric embedding.
\end{enumerate}
\end{prop}
\begin{proof}
As noted above, this proposition holds by~\cite[Lemma~2.5]{M11}. The proposition also follows directly from the older result of Swenson~\cite{Swe}.
Indeed, (2) obviously implies (1). Thus assume that (1) holds and that the Cannon-Thurston map $i:\partial G\to\partial Y$ exists and is injective. Since both $\partial G$ and $\partial Y$ are compact and Hausdorff, the map $i$ is a $G$-equivariant homeomorphism between $\partial G$ and $Z'=i(\partial G)$.  Since $G$ is word-hyperbolic, every point of $\partial G$ is conical. Therefore every point of $Z'$ is conical for the $G$-action on $Z'$ and hence, by Lemma~\ref{lem:subset}, also for the action of $G$ on $\partial Y$. The main result of Swenson~\cite{Swe} then implies that the orbit map $G\to Y$, $g\mapsto gp$, is a quasi-isometric embedding.
\end{proof}

Proposition~\ref{prop:inj} implies, in particular, that if $G_1$ is a word-hyperbolic subgroup of a word-hyperbolic group $G_2$ and if the Cannon-Thurston map $i:\partial G_1\to\partial G_2$ exists and is injective then $G_1$ is quasiconvex in $G_2$. 

\section{Algebraic laminations}

If $G$ is a word-hyperbolic group, we denote $\partial^2 G:=\{(z,s)\in \partial G\times \partial G| z\ne s\}$.  The set $\partial^2G$ is equipped with the subspace topology from the product topology on $\partial G\times\partial G$. The group $G$ has a natural diagonal action on $\partial^2 G$: for $g\in G$ and $(z,s)\in\partial^2G$ we have $g(z,s):=(gz,gs)$.
Let $\partial G\times\partial G\to\partial G\times\partial G$ be the ``flip" map given by $j: (x,y)\mapsto (y,x)$ for $(x,y)\in \partial G$.

\begin{defn}[Algebraic lamination]
Let $G$ be a word-hyperbolic group. An \emph{algebraic lamination} on $G$ is a subset $L\subseteq \partial^2 G$ such that $L$ is closed in $\partial^2 G$, flip-invariant and $G$-invariant.
A pair $(x,y)\in L$ is called a \emph{leaf} of $L$. An element $x\in \partial G$ is called an \emph{end} of $L$ if there exists $y\in \partial G$, $y\ne x$ such that $(x,y)\in L$.

For an algebraic lamination $L$ on $G$ denote by $End(L)$ the set of all ends of $L$. Note that $End(L)$ is a $G$-invariant subset of $\partial G$.
\end{defn}

\begin{defn}[Lamination and relation associated to a Cannon-Thurston map] \label{Def:CT lam and rel}

Let $G$ be a word-hyperbolic group and let $Z$ be a compact metrizable space equipped with a convergence action of $G$ such that the Cannon-Thurston map $i: \partial G\to  Z$ exists.
Denote \[L_i:=\{(x,y)\in \partial G\times\partial G | i(x)=i(y), x\ne y\}.\]  Since $i$ is continuous and $G$-equivariant,  $L_i$ is a closed $G$-invariant and flip-invariant subset of $\partial^2 G$. Thus $L_i$ is an algebraic lamination on $G$.
\end{defn}

\section{Dynamical characterization}

Suppose $G$ is a word-hyperbolic group acting as a non-elementary convergence group on $Z$, and let $i \colon \partial G \to Z$ be a Cannon--Thurston map.  We say that a point $x \in \partial G$ is {\em asymptotic to $L_i$} if for {\em every} conical sequence $\{g_n\}_{n=1}^\infty$ for $x$ with pole pair $(x_-,x_+)$ we have $(x_-,x_+)\in L_i$, that is, $i(x_-) = i(x_+)$.

In this section, we prove the first theorem from the introduction.

\begin{theoremA} Suppose $G$ is word-hyperbolic and acts on the compact, metrizable space $Z$ as a non-elementary convergence group, and suppose $i \colon \partial G \to Z$ is a Cannon--Thurston map.  Let $z\in i(\partial G)$. Then:
\begin{enumerate}
\item The point $z \in Z$ is not a conical limit point for the action of $G$ on $Z$  if and only if some point $x \in i^{-1}(z)$ is asymptotic to $L_i$.  

\item If $|i^{-1}(z)| > 1$, then any $x \in i^{-1}(z)$ is asymptotic to $L_i$, and hence $z$ is non-conical.
\end{enumerate}
\end{theoremA}

In the rest of this section, we make the assumptions of the theorem.  Our first lemma shows that, up to subsequences, conical sequences in $G$ for $Z$ must come from conical sequences for $\partial G$.
\begin{lem} \label{L:conical from conical}
Suppose that $z\in Z$ is a conical limit point and that $\gnn$ is a conical sequence for $z$.  Then there exists $x \in i^{-1}(z)$ and a subsequence $\gnkk$ which is a conical sequence for $x$.  Moreover, if $(x_-,x_+)$ is the pole pair for $\gnk$ and $x$, then $(i(x_-),i(x_+))$ is the pole pair for $\gnk$ and $z$, and in particular, $i(x_-) \neq i(x_+)$.
\end{lem}
\begin{proof}
Without loss of generality, we may assume that the action of $G$ on $Z$ is minimal, so $i(\partial G) = Z$.  Let $(z_-,z_+)$ be the pole pair for $z$ and $\gn$.  This is also a pole pair for any subsequence of $\gn$.

By the convergence property, there exists subsequence data $(x,x_-,g_{n_k})$ such that $(g_{n_k}|_{\partial G \setminus \{x \}} )_{k \geq 1})$ converges locally uniformly to the constant map to $x_-$.  By passing to a further subsequence, we may assume that $\displaystyle{\lim_{k \to \infty} g_{n_k}(x) = x_+}$, for some $x_+ \in \partial G$ (possibly equal to $x_-$).

Since $i$ is continuous, it follows that for any $y \in \partial G$, we have
\[ i(\lim_{k \to \infty} g_{n_k}(y)) = \lim_{k \to \infty} g_{n_k}(i(y)).\]
From this, we see that if $y \in \partial G \setminus ( i^{-1}(z) \cup \{x\})$ then $\displaystyle{i(x_-) =i(\lim_{k \to \infty} g_{n_k}(y))= z_-}$.
Furthermore, since any $y \in \partial G \setminus \{x\}$ has $\displaystyle{\lim_{k \to \infty} g_{n_k}(i(y)) = i(x_-) = z_-}$, it follows that $i(\partial G \setminus \{x \}) \subset Z \setminus \{z\}$; that is, $i(x) = z$.  Finally, we have
\[ i(x_+) = i(\lim_{k \to \infty} g_{n_k}(x)) = \lim_{k \to \infty} g_{n_k}(z) = z_+.\]
Therefore $i(x_+) = z_+ \neq z_- = i(x_-)$, and so $x_+ \neq x_-$ and $\{g_{n_k}\}$ is a conical sequence for $x$ with pole pair $(x_-,x_+)$.
\end{proof}

\begin{proof}[Proof of Theorem~\ref{thm:A}]
To prove part (1), first, suppose that $z \in Z$ is conical.  Let $\gnn$ be a conical sequence for $z$ with pole pair $(z_-,z_+)$.  According to Lemma \ref{L:conical from conical}, there exist $x \in i^{-1}(z)$ and a subsequence $\gnkk$ which is a conical sequence for $x$ with pole pair $(x_-,x_+)$.  Since $i(x_-) = z_- \neq z_+ = i(x_+)$, it follows that $x$ is not asymptotic to $L_i$.

Now suppose $x$ is not asymptotic to $L_i$ and let $\{g_n\}_{n=1}^{\infty}$  be any conical sequence for $x$ with pole pair $(x_-,x_+)$ such that $i(x_-) \neq i(x_+)$.  Because the action of $G$ on $Z$ is a convergence action, there exist a subsequence $\gnkk$ and $z,z_- \in Z$ such that on $Z \setminus \{z\}$, $g_{n_k}$ converges locally uniformly to $z_-$.  For any $y \in \partial G \setminus \{x\}$ we have $\displaystyle{\lim_{k \to \infty} g_{n_k}(i(y)) = i(x_-)}$.  Thus taking $y \not \in i^{-1}(z)$, this implies $i(x_-) = z_-$.  On the other hand, $\displaystyle{\lim_{k \to \infty} i(g_{n_k}(x)) = i(x_+) \neq i(x_-) = z_-}$ by assumption.  It follows that $i(x) = z$ (since anything else must converge to $z_-$ on applying $g_{n_k}$). Therefore, setting $z_+ = i(x_+)$, it follows that $\{g_{n_k}\}$ is a conical sequence for $z$ with pole pair $(z_-,z_+)$, and $z$ is a conical limit point. Thus part (1) of Theorem~\ref{thm:A} is proved. 

For part (2) of Theorem~\ref{thm:A} we suppose $|i^{-1}(z)| >1$, and prove that any $x \in i^{-1}(z)$ is asymptotic to $L_i$.   For this, let $y \in i^{-1}(z)$ be any other point with $y \neq x$.  Let $\gnn$ be a conical sequence for $x$ with pole pair $(x_-,x_+)$.  Then $\limn g_n(x) = x_+$ and $\limn g_n(y) = x_-$.    Since $L_i$ is $G$--invariant, $i(g_n(x)) = i(g_n(y))$ and since $L_i$ is closed (or equivalently, the algebraic lamination $L_i$ is closed), it follows that $i(x_-) = i(x_+)$.  Since $\gn$ was an arbitrary conical sequence for $x$, the point $x$ is asymptotic to $L_i$, as required. Hence, by part (1), $z$ is not a conical limit point for the action of $G$ on $Z$. 
\end{proof}

\section{Geometric characterization}\label{sect:geom}

\begin{defn}[Coarse leaf segments]
Let $G$ be a word-hyperbolic group and let $L\subseteq \partial^2 G$ be an algebraic lamination on $G$.

Let $X$ be a $\delta$-hyperbolic (where $\delta\ge 1$) proper geodesic metric space equipped with a properly discontinuous cocompact isometric action of $G$, so that $\partial G$ is naturally identified with $\partial X$.

For an algebraic lamination $L$ on $G$, a geodesic segment $\tau=[a,b]$ in $X$ is called a \emph{coarse $X$-leaf segment of $L$} if there exist a pair $(x,y)\in L$ and a bi-infinite geodesic $\gamma$ from $x$ to $y$ in $X$ such that $\tau$ is contained in the $2\delta$-neighborhood of $\gamma$.
\end{defn}

If $C\ge 0$, for a geodesic segment $\tau=[a,b]$ of length $\ge 2C$, the \emph{$C$-truncation} of $\tau$ is defined as $[a',b']\subseteq [a,b]$ where $a',b'\in [a,b]$ are such that $d(a,a')=d(b,b')=C$.

\begin{theoremB} Let $G$ be a word-hyperbolic group and let $Z$ be a compact metrizable space equipped with a non-elementary convergence action of $G$ such that the Cannon-Thurston map $i: \partial G\to  Z$ exists and such that $i$ is not injective. Let $X$ be a $\delta$-hyperbolic (where $\delta\ge 1$) proper geodesic metric space equipped with a properly discontinuous cocompact isometric action of $G$ (so that $\partial G$ is naturally identified with $\partial X$).

Let $x\in \partial G$, let $z=i(x)\in Z$ and let $\rho$ be a geodesic ray in $X$ limiting to $x$.

Then the following are equivalent:
\begin{enumerate}
\item The point $z$ is a  conical limit point for the action of $G$ on $Z$.
\item There exist a geodesic segment $\tau=[a,b]$ in $X$ of length $\ge 100\delta$ and an infinite sequence of distinct elements $g_n\in G$ such that the $20\delta$-truncation $\tau'$ of $\tau$ is not a coarse $X$-leaf segment of $L_i$ and such that 
for each $n\ge 1$ the segment $g_n\tau$ is contained in a $6\delta$-neighborhood of $\rho$. [Note that this condition automatically implies that $\lim_{n\to\infty} g_na=\lim_{n\to\infty} g_nb=x$.]
\end{enumerate}
\end{theoremB}

\begin{proof}
Suppose first that (1) holds and that $z$ is a conical limit point for the action of $G$ on $Z$.   Since by assumption $i$ is not injective, there exists a pair $(y',y)\in L_i$ such that $i(y)=i(y')$.  
Denote $s=i(y)=i(y')$.
By translating by an element of $g$ if necessary, we may also assume that $s\ne z$.

Since $i(x)=z$ and $z\ne s$, we have $x\ne y$. Note that $y\in End(L_i)$.

Consider a geodesic $\gamma$ from $y$ to $x$ in $X$.  Since $z$ is conical, by Proposition~\ref{prop:bow} there exists an infinite sequence of distinct elements $h_n\in G$ such that $\lim_{n\to\infty}h_n(s,z)=(s_\infty,z_\infty)$ for some $s_\infty,z_\infty\in Z$ such that $s_\infty\ne z_\infty$.  After passing to a further subsequence, we may assume that $\lim_{n\to\infty} h_n x=x_\infty$ and $\lim_{n\to\infty} h_n y=y_\infty$ for some $x_\infty, y_\infty\in \partial G=\partial X$. By continuity of $i$ we have $i(x_\infty)=z_\infty$ and $i(y_\infty)=s_\infty$. In particular, this means that $x_\infty\ne y_\infty$ and that $\lim_{n\to\infty} h_n(y,x)=(y_\infty,x_\infty)$ in $\partial^2 G$.  Let $\gamma_\infty$ be a geodesic in $X$ from $y_\infty$ to $x_\infty$.

 Then there exists a sequence of finite subsegments $\tau_n=[q_n, r_n]$ of $\gamma$ and a sequence of subsegments $[a_n,b_n]$ of $\gamma_\infty$ with the following properties:

\renewcommand{\labelenumi}{\alph{enumi})}
\begin{enumerate}
\item We have $\lim_{n\to\infty} a_n=y_\infty$,  $\lim_{n\to\infty} b_n=x_\infty$ and $[a_n, b_n]$ is a subsegment of $[a_{n+1}, b_{n+1}]$.
\item We have either $\lim_{n\to\infty} q_n=\lim_{n\to\infty} r_n=x$ or $\lim_{n\to\infty} q_n=\lim_{n\to\infty} r_n=y$;
\item For all $n\ge 1$ the paths $h_n[q_n,r_n]$ and $[a_n,b_n]$ are $4\delta$-close.
\item We have $h_nq_n\to_{n\to\infty} y_\infty$ and $h_nr_n\to_{n\to\infty} x_\infty$.
\end{enumerate}

If  $\lim_{n\to\infty} q_n=\lim_{n\to\infty} r_n=y$ then, since $y\in End(L_i)$ and since $L_i\subseteq \partial^2 G$ is closed, it follows that $(y_\infty,x_\infty)\in L_i$. 
Therefore  $z_\infty=i(x_\infty)=i(y_\infty)=s_\infty$, which contradicts the fact that $s_\infty\ne p_\infty$. 
Therefore $\lim_{n\to\infty} q_n=\lim_{n\to\infty} r_n=x$. Since $s_\infty\ne z_\infty$,  it follows that $(y_\infty,x_\infty)\not\in L_i$.   Then
there exists $m\ge 1$ such that $d(a_m,b_m)\ge 100\delta$ and such that for $\tau:=[a_m,b_m]$  the $20\delta$-truncation $\tau'=[a_m',b_m']\subseteq \gamma_\infty$ of $\tau$ is not a coarse $X$-leaf segment of $L_i$. 
By construction, for every $n\ge m$, $h_n^{-1}\tau$ is contained in a $4\delta$-neighborhood of $[q_n,r_n]$ and hence, for all sufficiently large $n$,  in a $6\delta$-neighborhood of $\rho$. Thus we have verified that (1) implies (2).

Suppose now that (2) holds and that there exist a geodesic segment $\tau=[a,b]$ in $X$ of length $\ge 100\delta$ and an infinite sequence of distinct elements $g_n\in G$ such that the $10\delta$-truncation $\tau'=[a',b']$ of $\tau$  is not an $X$-leaf segment of $L_i$ and such that 
for each $n\ge 1$ the segment $g_n\tau$ is contained in a $6\delta$-neighborhood of $\rho$. 
We claim that $z$ is a conical limit point for the action of $G$ on $Z$. In view of Lemma~\ref{lem:subset},  we may assume that $i(\partial G)=Z$.

Indeed, let $s\in Z$ be arbitrary  such that $s\ne z$. Recall that $i(x)=z$. Choose $y\in \partial G$ such that $i(y)=s$. Thus $x\ne y$.
Consider the bi-infinite geodesic $\gamma$ from $y$ to $x$ in $X$. Recall that $\rho$ is a geodesic ray in $X$ limiting to $x$.

After chopping-off a finite initial segment of $\rho$ if necessary, we may assume that there is a point $w\in \gamma$ such that the ray $\rho'$  from $w$ to $x$ contained in $\gamma$ is $2\delta$-close to $\rho$. By assumption, for every $n\ge 1$ the geodesic $g_n^{-1} \gamma$ from $g_n^{-1}y$ to $g_n^{-1}x$ contains a subsegment which is $8\delta$-close to $\tau$.  By compactness, after  passing to a further subsequence, we may assume that $\lim_{n\to\infty} g_n^{-1} y=y_\infty$ and $\lim_{n\to\infty} g_n^{-1} x=x_\infty$ for some distinct points $x_\infty, y_\infty\in \partial G$.
Let $\gamma_\infty$ be a geodesic from $y_\infty$ to $x_\infty$ in $X$. 

We have $\tau'=[a',b']\subseteq [a,b]=\tau$ with $d(a,a')=d(b,b')=20\delta$.  Since $\tau$ is contained in the $8\delta$-neighborhood of  $g_n^{-1} \gamma$, the segment $\tau'$ is contained in a $2\delta$-neighborhood of $\gamma_\infty$, and $\tau'$ has length $\ge 50\delta$. Since by assumption $\tau'$ is not a coarse $X$-leaf segment of $L_i$, it follows that $(y_\infty,x_\infty)\not\in L_i$ and hence $i(x_\infty)\ne i(y_\infty)$.  Denote $z_\infty=i(x_\infty)$ and $s_\infty=i(y_\infty)$. Since $i(x)=z$, $i(y)=s$ and since $\lim_{n\to\infty} g_n^{-1} y=y_\infty$ and $\lim_{n\to\infty} g_n^{-1} x=x_\infty$,  the continuity of $i$ implies that $\lim_{n\to\infty} g_n^{-1}(s,z)=(s_\infty,z_\infty)$. Since $s_\infty\ne z_\infty$, Proposition~\ref{prop:bow} implies that $z$ is indeed a conical limit point for the action of $G$ on $Z$, as required.

\end{proof}

\section{Injective, non-conical limit points}

Here we prove that injective non-conical limit points occur quite often.

\begin{theoremC}
Suppose a hyperbolic group $G$ acts on a compact metrizable space $Z$ as a non-elementary convergence group without accidental parabolics, and suppose that there exists a non-injective Cannon--Thurston map $i \colon \partial G \to Z$.  Then there exists a non-conical limit point $z \in Z$ with $|i^{-1}(z)| = 1$.
\end{theoremC}

Suppose that $G$ is a hyperbolic group acting as a non-elementary convergence group on $Z$ as in the statement of the theorem, from which it follows that $G$ is also non-elementary.  Fix a finite generating set $\mathcal S$ for $G$, such that $\mathcal S=\mathcal S^{-1}$, and let $X$ be the Cayley graph of $G$ with respect to $\mathcal S$, endowed with the usual geodesic metric in which every edge has length $1$.  Then $X$ is $\delta$--hyperbolic for some $\delta > 0$. We denote the length of a geodesic segment $\sigma$ in $X$ as $\vert \sigma \vert$.  Recall that for $r>0$ an \emph{$r$--local geodesic} in $X$ is a path $\alpha$ parameterized by arclength such that every subsegment of $\alpha$ of length $r$ is a geodesic.
There exist integers $r > 0$ and $D > 0$ such that any $r$--local geodesic in $X$ is quasi-geodesic (with constants depending only on $r$ and $\delta$), and such that the Hausdorff distance between an $r$--local geodesic and the geodesic with the same endpoints is at most $D$; see e.g.~\cite[Part III, Chapter 1]{BridHae}.

Given an algebraic lamination $L \subset \partial^2 G$, define the {\em geodesic realization of $L$ with respect to $\mathcal S$}, denoted  $\mathcal L$, as the set of all $\ell \subset X$ such that there exist $x,y\in \partial G$ with $(x,y) \in L$ such that $\ell$ is a bi-infinite geodesic in $X$ from $x$ to $y$.
 
\begin{conv}
For the remainder of this section, we assume $G,Z,i$ are as in the statement of the theorem, ${\mathcal S},X,\delta,r,D$ are as above, let $L_i$ be the algebraic lamination associated to $i$ as in Definition~\ref{Def:CT lam and rel}, and let $\mathcal L_i$ denote the geodesic realization of $L_i$.
\end{conv}


Given integers $p \geq 1$, a \emph{$p$--periodic, $r$--local geodesic in $X$} is a bi-infinite $r$--local geodesic $\gamma$ in $X$ for which some element $g \in G$ acts on $\gamma$ translating a distance $p$ along $\gamma$.  As $\gamma$ is a quasi-geodesic, it follows that $g$ has infinite order (and $\gamma$ is a quasi-geodesic axis for $g$ in $X$).

We will use the following lemma in the proof of the theorem.
\begin{lem} \label{L:poison power}
For any $p\ge 1$, there exists $c(p)\ge 1$ with the following property.  If $\gamma$ is a $p$--periodic, $r$--local geodesic in $X$ and $\ell \in \mathcal L_i$ contains a segment $\sigma \subset \gamma$ in its $\delta + D$ neighborhood, then $|\sigma| < c(p)$.
\end{lem}

\begin{proof}
It suffices to prove this statement for any fixed $p$--periodic, $r$--local geodesic $\gamma$ in $X$ (since, for a given $p$,  there are only finitely many $G$--orbits of such $\gamma$).  Translating such $\gamma$ if necessary, we may assume that $\gamma$ passes through the identity $1$ in $G \subset X$.  Let $g \in G$ be a translation of length $p$ along $\gamma$.

Now if the requisite $c(p)$ does not exist, then there exists a sequence $\{ \ell_n \}_{n \geq 1}$ of elements of $\mathcal L_i$ so that each $\ell_n$ contains a segment $\sigma_n \subseteq \gamma$ of length at least $n$ in its $\delta + D$--neighborhood.  Since $\mathcal L_i$ is $G$--invariant, after applying an appropriate power of $g$ to $\ell_n$ if necessary, we can assume that the midpoint of $\sigma_n$ lies within distance $p$ of $1 \in G$.  In particular, for $n > p$, $1 \in \sigma_n$ and $\ell_n$ is within $\delta + D$ of $1$. Passing to a subsequence, we can assume that $\ell_n \to \ell \in \mathcal L_i$ as $n \to \infty$ (since $L_i$ is closed).  On the other hand, since $\sigma_n \to \gamma$, as $n \to \infty$, we see that $\gamma$ is within $\delta + D$ of $\ell$.  Therefore, $\ell$ and $\gamma$ have the same endpoints on $\partial G$. Since the endpoints of $\gamma$ are the fixed points of $g$, and the endpoints of $\ell$ are identified by $i$, it follows that $g$ is an accidental parabolic for the action on $Z$, yielding a contradiction.
\end{proof}

We are now ready for the proof of the theorem.
\begin{proof}[Proof of Theorem \ref{thm:C}.]
Let $\ell \in \mathcal L_i$ be a bi-infinite geodesic in $\mathcal L_i$ and $\ell_+ \subset \ell$ be a geodesic ray contained in $\ell$.  We can view $\ell_+$ as a semi-infinite word over the alphabet $\mathcal S$. 

For any $m \geq r$, let $v_m \in \mathcal S^\ast$ be a word of the length $m$ which occurs (positively) infinitely often in $\ell_+$.  Such a word exists, for every $m$, by the pigeonhole principal.  Now we define several additional families of sub-words of $\ell_+$.  These subwords will serve as the building blocks for a new $r$--local geodesic (hence quasigeodesic) infinite ray.

For each $m \geq r$
\begin{enumerate}
\item let $u_m$ be any subword of $\ell_+$ of length at least $m$ so that $v_mu_mv_m$ occurs in $\ell_+$.  Such $u_m$ exists because $v_m$ occurs in $\ell_+$ infinitely often;

\item let $t_m$ be any nonempty word so that $v_mt_mv_{m+1}$ occurs in $\ell_+$.  These exist for the same reason as $u_m$;
\item put $\alpha_m = v_mu_m$.
\end{enumerate}

Let $p_m = |\alpha_m|$, and let $\kappa_m > 0$ be an integer such that $\kappa_m p_m > c(p_m)$.
Note that, since $v_mu_mv_m$ is a subword of a geodesic ray $\ell_+$ with $|u_m|,|v_m|\ge m$,  it follows that the word $\alpha_m=v_mu_m$ is cyclically reduced and that for every $k\ge 1$ every subword of length $m$ in $\alpha_m^k$ is a geodesic and occurs as a subword of $v_mu_mv_m$ and thus of $\ell_+$.

Now consider the following semi-infinite word (which we also view as a semi-infinite path in $X$ with origin $1\in G$):
\[ w_\infty := \alpha_r^{\kappa_r}v_rt_r \alpha_{r+1}^{\kappa_{r+1}} v_{r+1}t_{r+1}\alpha_{r+2}^{\kappa_{r+2}} \cdots.\]
This word $w_\infty$ is naturally a union of subwords of the following forms:
\begin{enumerate}
\item $\alpha_m$, which is a subword of $\ell_+$;
\item $v_mt_mv_{m+1}$, which is a subword of $\ell_+$.
\item $\alpha_m^{\kappa_m}$, which is a word of length $p_m\kappa_m> c(p_m)$, is contained in a $p_m$--periodic, $r$--local geodesic.  As such, the word $\alpha_m^{\kappa_m}$ is not contained in a $D+\delta$-neighborhood of any  $\ell' \in \mathcal L_i$, by Lemma \ref{L:poison power}.  However, any subword of length $m$ of $\alpha_m^{\kappa_m}$ occurs in $\ell_+$.
\end{enumerate}
Moreover, any subword $v$ of $w_\infty$ of length $r$ is contained in at least one such word, and thus $v$ occurs as a subword of $\ell_+$. Therefore $w_\infty$ is an $r$--local geodesic in $X$ and hence a global quasigeodesic in $X$. Furthermore, note that as $m$ tends toward infinity, the lengths of the words $\alpha_m$ and $v_mt_mv_{m+1}$  tend to infinity.  Denote the endpoint of $w_\infty$ in $\partial G$ by $x$. 

First, we claim that $|i^{-1}(i(x))| = 1$.  If this were not the case, then the ray $w_\infty$  would be asymptotic to (i.e. have a finite Hausdorff distance to) an infinite ray $\ell_+' \subset \ell'$ for some geodesic $\ell' \in \mathcal L_i$.  In this case, a subray $w_\infty' \subset w_\infty$ would be contained in the $\delta + D$ neighborhood of $\ell_+'$.  Since this ray contains arcs labeled $\alpha_m^{\kappa_m}$ for $m$ sufficiently large, this contradicts Lemma \ref{L:poison power}.

Second, we claim that $i(x)$ is non-conical.   To prove this, let $\gamma$ be an $r$--local geodesic containing $w_\infty$ as a subray.  For example, let $\gamma$ be the concatenation of the ray which is $\alpha_r^{-\infty}$ with $w_\infty$.  One endpoint of $\gamma$ is $x$, and we denote the other by $y$.  Let $(g_n)_{n \ge 1}$ be any convergence sequence for $x$ with pole pair $(x_-,x_+)$.  Then $g_n(x) \to x_+$ and $g_n(y) \to x_-$.   Since $x_- \neq x_+$, after passing to a subsequence $\gnk$, the $r$--local geodesics $g_{n_k} \gamma$ must converge to an $r$--local geodesic with endpoints $x_-,x_+$.  After passing to a further subsequence (still denoted $\gnk$), it follows that $g_{n_k} \gamma$ converges to an $r$--local geodesic with endpoints $x_-,x_+$.  Since $\gnk$ is a convergence sequence for $x$, if $k$ is sufficiently large, any closest point $h_k \in g_{n_k} \gamma$ to $1$ must have $g_{n_k}^{-1}(h_k) \in w_\infty$ with distance to the initial point of $w_\infty$ tending toward infinity.  Passing to yet a further subsequence if necessary, we can assume that the subword $w_k \subset w_\infty$ of length $2k$ centered on $g_{n_k}^{-1}h_k$ is a subword of $\ell_+ \subset \ell$.  Thus for all $k$ there exists $\ell_k \in \mathcal L_i$ (a translate of $\ell$) so that the {\em segment} $w_k$ is contained in $\ell_k$.

Now observe that $g_{n_k}(w_k)$ is a segment of $g_{n_k}(\ell_k) \in \mathcal L_i$.  Because $g_{n_k}(w_k)$ is a geodesic of length $2k$ centered on $h_k$, it follows that $g_{n_k}(w_k)$, and hence $g_{n_k}(\ell_k)$, converges to a geodesic with endpoints $(x_-,x_+)$ as $k \to \infty$.  However, $g_{n_k}(\ell_k)$ must converge to a leaf of $\mathcal L_i$ since $L_i$ is closed.  Since $\gn$ was an arbitrary convergence sequence for $x$, $x$ is asymptotic to $L_i$, and by Theorem \ref{thm:A}, $i(x)$ is non-conical.
 \end{proof}

\section{Controlled concentration points}

\begin{defn}\label{d:controlled} 
Let $G$ be a non-elementary torsion-free discrete subgroup of hyperbolic isometries acting on $\Hy^n$ and let $S^{n-1}_\infty$ be the ideal boundary of $\Hy^n$.
A neighborhood $U\subset S^{n-1}_\infty$ of $x\in \Lambda(G)$ is called {\it concentrated} at $x$ if for every neighborhood $V$ of $x$, there exists an element $g\in G$ such that $x\in g(U)$ and $g(U)\subset V$.
If such $g$ can always be chosen so that $x\in g(V)$ then we say $U$ is {\it concentrated with control}.
A limit point $x$ in $\Lambda(G)$ is called a {\it controlled concentration point} if it has a neighborhood which is concentrated with control.
\end{defn}

A geodesic ray in $\Hy^n$ is called {\it recurrent with respect to $G$} if its image $\alpha$ in $M=\Hy^n/G$ by the covering projection is recurrent. Recall that a geodesic ray $\alpha$ parametrized by $[0,\infty)$ in $M$ is called {\it recurrent} if for any tangent vector $v=\alpha'(t_0), t_0>0$ in the unit tangent bundle $UT(M)$ of $M$, there exists an infinite sequence of times $\{t_i\}$ such that $\alpha'(t_i)$ converges to $v$ in $UT(M)$. 
The main result of \cite{AHM} is that controlled concentration points correspond to the end points of recurrent geodesic rays.
\begin{thm}[Aebischer, Hong and McCullough \cite{AHM}]
\label{CC} Let $G$ be as in Definition~\ref{d:controlled}.  Then for a limit point $x \in \Lambda(G)$, the following are equivalent: 
\begin{enumerate}
\item There is a recurrent geodesic ray whose endpoint is $x$.
\item $x$ is a controlled concentration point.
\item There exists a sequence $\{g_n\}$ of distinct elements of $G$ such that for any geodesic ray $\beta$ whose endpoint $\beta(\infty)$ is $x$, $g_n(\beta)$ converges to some geodesic ray whose endpoint is again $x$ up to taking a subsequence.
\item There exists a sequence $\{g_n\}$ of distinct elements of $G$  and $y\in S^{n-1}_\infty$ with $y\neq x$ such that $g_n x \rightarrow x$ and $g_n\vert_{S^n_\infty\setminus\{x\}}$ converges uniformly on compact subsets to the constant map to $y$.
\end{enumerate}
\end{thm} 

From the last characterization of a controlled concentration point, it is clear that every controlled concentration point is conical, but the converse is not true in general.
In fact \cite[Prop. 5.1]{AHM} gives an example of a conical limit point which is not a controlled concentration point in the case of a rank-$2$ Schottky group. 
It is also known that the set of controlled concentration points has full Patterson-Sullivan measure in $\Lambda(G)$ if $G$ is of divergence type. 
Note that for a geodesic lamination $\lambda$ in a hyperbolic surface $S$, a leaf of $\lambda$ is always a recurrent geodesic, and hence its endpoints are controlled concentration points.

The following proposition follows easily from the condition (1) of Theorem \ref{CC}.
\begin{prop}
A limit point $x$ in $\Lambda(G)\subset S^{n-1}_\infty$ is a controlled concentration point if  there exists a geodesic ray $\beta$ in $\Hy^n$ which limits to $x$ and the $\omega$-limit set of $\beta$ in the geodesic foliation on the unit tangent bundle $UT(M)$ of $M=\Hy^n/G$ has only one minimal component.
\end{prop}

We extend the notion of controlled concentration points to the case of hyperbolic groups.
\begin{defn}\label{defn:controlled}
Let $G$ be a non-elementary hyperbolic group. Then we say $x \in \partial G$ is a \emph{controlled concentration point} if 
there exists a sequence $\{g_n\}$ of distinct elements of $G$  and $y\in \partial G$ with $y\neq x$ such that $g_n x \rightarrow x$ and $g_n\vert_{\partial G\setminus\{x\}}$ converges locally uniformly to the constant map to $y$.
\end{defn}

\begin{prop}\label{propCC}
Suppose $G$ is word-hyperbolic and acts on the compact, metrizable space $Z$ as a non-elementary convergence group, and suppose $i \colon \partial G \to Z$ is a Cannon-Thurston map. If a controlled concentrated point $x \in \partial G$  satisfies
$\vert i^{-1}(i(x))\vert=1$ then $i(x)\in Z$ is conical. 
\end{prop}
\begin{proof}
Since $x$ is a controlled concentrated point, there exist $y\in \partial G$ with $y\neq x$ and a sequence $\{g_n\}$ of distinct elements in $G$ such that
$\limn g_n x =x$ and $(g_n \vert_{\partial G\setminus \{x\}})$ locally uniformly converges to the constant map to $y$.
Note that $\vert i^{-1}(i(x))\vert=1$ implies $i(x)\neq i(y)$.
Suppose $i(x)\in Z$ is not conical. Then by Proposition~\ref{prop:bow} there exists a subsequence $(g_{n_k})$ of $(g_n)$ such that $\limk g_{n_k}i(x) = \limk g_{n_k}i(y)$ and hence by continuity
\[ (i(x),i(y)) = \limk (i(g_{n_k}x), i(g_{n_k}y))= \limk (g_{n_k}i(x), g_{n_k}i(y)) = (z,z) \]
for some $z\in Z$, and hence $i(x) = i(y)$.  This is a contradiction.
\end{proof}

We can now prove the last theorem  from Introduction:

\begin{theoremD}
Let $G$ be a non-elementary torsion-free word-hyperbolic group. Then there exists $x\in \partial G$ which is not a controlled concentration point.
\end{theoremD}
\begin{proof}
Kapovich~\cite{Ka99} proved that, given a non-elementary  torsion-free word-hyperbolic group $G$, there exists a word-hyperbolic group $G_\ast$ containing $G$ as a non-quasiconvex subgroup.
Moreover,  $G_\ast$ is constructed in \cite{Ka99} as an HNN-extension \[G_\ast=\langle G, t | t^{-1} K t= K_1\rangle\] where $K\le G$ is a quasiconvex free subgroup of rank $2$ and where $K_1\le K$ is also free of rank $2$ (and hence $K_1$ is also quasiconvex in $G$). Therefore, by a general result of Mitra~\cite{M4} (see also \cite{MP11}) about graphs of groups with hyperbolic edge and vertex groups, there does exist a Cannon-Thurston map $i:\partial G\to\partial G_\ast$. Since $G\le G_\ast$ is not quasiconvex, Proposition~\ref{prop:inj} implies that the map $i$ is not injective.  Therefore,  by Theorem \ref{thm:C}, there exists a non-conical limit point $z\in i(\partial G)$ with $\vert i^{-1}(z)\vert=1$. By Proposition~\ref{propCC}, $x=i^{-1}(z)\in \partial G$ is not a controlled concentration point.
\end{proof}



\appendix
\section{Descriptions of $L_i$}\label{s:appendix}

There are several situations where the Cannon-Thurston map $i:\partial G\to Z$ is known to exist and where a more detailed description of the lamination $L_i$ is known.  Theorem~\ref{thm:A} and Theorem~\ref{thm:B} may be useful in these contexts. The proof of Theorem~6.6 in \cite{DKT} uses Theorem~\ref{thm:B} as a key ingredient in this way.

\subsection{Kleinian representations of surface groups}
Let $G$ be the fundamental group of a closed, orientable hyperbolic surface $S$.   The universal covering of $S$ is isometric to the hyperbolic plane $\Hy^2$ with $G$ acting cocompactly by isometries, and so we can identify the Gromov boundary of $G$ with the circle at infinity $\partial G \cong S^1_\infty$.  A faithful {\em Kleinian representation} $\rho \colon G \to \PSL(2,\mathbb C) \cong {\rm Isom}^+(\Hy^3)$ is an injective homomorphism with discrete image.  This determines a convergence action of $G$ on $S^2_\infty$, and hence also on the limit set $\Lambda(G)$.  The existence of a Cannon--Thurston map for such groups was first proved in the special case when $\rho(G)$ is the fiber subgroup of a  hyperbolic $3$--manifold fibering over the circle by Cannon and Thurston \cite{CT}.  This was extended to include other classes of Kleinian representations of $G$ in \cite{Min94, M4} and then arbitrary faithful, Kleinian representation of $G$ in \cite{M14}.

The hyperbolic 3-manifold $M=\Hy^3/\rho(G)$ is homeomorphic to $S\times (-\infty, \infty)$ by the Tameness Theorem (\cite{Bon}, and \cite{Ag, CG} in more general settings) and thus $M$ has only two ends, $E_+$ and $E_-$.  
Assume that $\rho(G)$ has no parabolics.
Associated to each end is a (possibly empty) {\em ending lamination} $\lambda_+$ and $\lambda_-$, which is a geodesic lamination on $S$, that is a closed union of pairwise disjoint complete geodesics; see \cite{CB} for more on geodesic laminations and \cite{Th,ELC1,ELC2} for more on the ending laminations associated to ends of $3$--manifolds.  
The preimage $\widetilde \lambda_\pm \subset \Hy^2$ of the ending laminations in $\Hy^2$ are geodesic laminations in $\Hy^2$, and the endpoints of the leaves determine a pair of algebraic laminations $L_\pm \subset \partial^2 G$.  Set $\mathcal R_+, \mathcal R_- \subset \partial G \times \partial G$ to be the reflexive, and transitive closures of the pair $L_+,L_-$, respectively.  
Then for the Cannon--Thurston map $i$ has $i(x) = i(y)$ if and only if $(x,y) \in \mathcal R_+ \cup \mathcal R_-$ according to \cite{CT} in the original setting, \cite{Min94} the cases treated there, and in general in \cite{M07}.    Furthermore, the transitive closure adds only endpoints of finitely many $G$--orbits of leaves, and thus $L_i$ is equal to $L_+ \cup L_-$, together with finitely many additional $G$--orbits of leaves (which correspond to the ``diagonals'' of the complementary components of $\widetilde \lambda_+$ and $\widetilde \lambda_-$).

\subsection{Short exact sequences of hyperbolic groups}

Let \[1\to G_1\to G_2\to Q \tag{$\ddag$}\] be a short exact sequence of three word-hyperbolic groups, such that $G_1$ is non-elementary. In this case $G_1$ acts on $Z=\partial G_2$ as a non-elementary convergence group without accidental parabolics. Mitra~\cite{M2} proved that in this case the Cannon-Thurston map $i:\partial G_1\to \partial G_2$ does exist. Therefore the results of this paper, including Theorem~\ref{thm:B}, do apply. 
In \cite{M1} Mitra also obtained a general geometric description of $L_i$ in this case in terms of the so-called ``ending laminations".

We give here a brief description of the results of \cite{M1}.

Given every $\xi\in\partial Q$, Mitra defines an ``ending lamination''
$\Lambda_\xi\subseteq \partial^2 G_1$. To define $\Lambda_\xi$, Mitra starts with choosing a
quasi-isometric section $r: Q\to G_2$ (he later proves that the specific
choice of $r$ does not matter). Then given any $\xi\in \partial Q$, take
a geodesic ray in $Q$ towards $\xi$ and let $\xi_n$ be the point at
distance $n$ from the origin on that ray. Lift $\xi_n$ to $G_2$ by the
section $r$ to get an element $g_n=r(\xi_n)\in G_2$. Conjugation by $g_n$
gives an automorphism $\phi_n$ of $G_1$ defined as
$\phi_n(h)=g_nhg_n^{-1}$, $h\in G_1$. Now pick any non-torsion element
$h\in G_1$. Then look at all $(x,y)\in \partial^2 G_1$ such that there exists a sequence of integers $k_n\to\infty$ and of conjugates (with respect to conjugation in $G_1$) $w_n$ of $\phi_{k_n}(h)$  in $G_1$ such that $\lim_{n\to\infty} (w_n^{-\infty}, w_n^\infty)=(x,y)$ in $\partial^2 G_1$.
 For a fixed non-torsion $h\in G_1$,
the collection of all such $(x,y)\in \partial^2 G_1$ is denoted
$\Lambda_{\xi,h}$.  Denote by $A$ the set of all elements of infinite order in $G_1$.
Finally, put $\Lambda_\xi=\cup_{h\in A} \Lambda_{\xi,h}$.
The main result of Mitra in
\cite{M1} says that,  in this case \[L_i=\cup_{\xi\in \partial Q} \Lambda_\xi.\]
For every $\xi\in\partial Q$ the subset $\Lambda_\xi\subseteq \partial^2 G_1$ is an algebraic lamination on $G_1$, and Mitra refers to $\Lambda_\xi$ as the "ending lamination" on $G_1$ corresponding to $\xi$. Moreover, the arguments of Mitra actually imply that if $\xi_1,\xi_2\in\partial Q$ are distinct, then $End(\Lambda_{\xi_1})\cap End(\Lambda_{\xi_2})=\varnothing$. Mitra also notes that for any $\xi\in\partial Q$ there exists a finite subset $B\subseteq A$ such that  $\Lambda_\xi=\cup_{h\in B} \Lambda_{\xi,h}$.

In general, for a short exact sequence $(\ddag)$ and $\xi\in\partial Q$, the ``ending lamination" $\Lambda_\xi\subseteq \partial^2 G_1$ can, at least a priori, be quite large and difficult to understand. This is the case even if $Q=\langle t\rangle\cong \Z$ is infinite cyclic, so that $\partial Q=\{t^{-\infty}, t^{\infty}\}$ consists of just two points.  However, in some situations the laminations $\Lambda_\xi$ are well-understood.

\subsection{Hyperbolic extensions of free groups}

In particular, let $N\ge 3$, let  $\phi\in \Out(F_N)$  be a fully irreducible atoroidal element and let $\Phi\in \Aut(F_N)$ be a representative of the outer automorphism class of $\phi$ (see \cite{BH92,BFH97,BFH00,CH12,HM11,HM13,Ka,LL} for the relevant background). Then 
\[
G=F_N\rtimes_\Phi \Z=\langle F_N, t| t w t^{-1}=\Phi(w), w\in F_N\rangle
\] is word-hyperbolic and we have a short exact sequence $1\to F_N\to G\to \langle t\rangle\to 1$.  Thus, by~\cite{M2}, there does exist a Cannon-Thurston map $i:\partial F_N\to\partial G$.
Using the results of Mitra~\cite{M1} mentioned above as a starting point, Kapovich and Lustig proved in \cite{KL6} that $\Lambda_{t^\infty}=diag(L_{BFH}(\phi))=L(T_-)$ and, similarly,  $\Lambda_{t^{-\infty}}=diag(L_{BFH}(\phi^{-1}))=L(T_+)$.  Here $L_{BFH}(\phi)\subseteq \partial^2 F_N$ is the \lq\lq stable" lamination of $\phi$, introduced by Bestvina, Feighn and Handel in \cite{BFH97}, and $diag(L_{BFH}(\phi^{-1}))$ is the \lq\lq diagonal extension" of $L_{BFH}(\phi)$, that is, the intersection of $\partial^2 F_N$ with the equivalence relation on $\partial F_N$ generated by the relation $L_{BFH}(\phi)\subseteq \partial^2 F_N$ on $\partial F_N$.  Also, here $L(T_-)$  is the "dual algebraic lamination" (in the sense of \cite{CHL,CHL2,KL3}) corresponding to the "repelling"  $\mathbb R$-tree $T_-$ for $\phi$ (the tree $T_-$ is constructed using a train-track representative for $\phi^{-1}$ and the projective class of $T_-$ is the unique repelling fixed point for the right action of $\phi$ on the compactified Outer space.
Thus, in view of the discussion above, we have 
\[
L_i=diag(L_{BFH}(\phi))\cup diag(L_{BFH}(\phi^{-1}))=L(T_-)\cup L(T_+)
\]
in this case. 
The stable lamination $L_{BFH}(\phi)$ of $\phi$ is defined quite explicitly in terms of a train-track representative $f:\Gamma\to\Gamma$ of $\phi$.  Thus a pair $(x,y)\in \partial^2 F_N$ belongs to $L_{BFH}(\phi)$ if and only if for every finite subpath $\tilde v$ of the geodesic from $x$ to $y$ in $\widetilde\Gamma$, the projection $v$ of $\tilde v$ to $\Gamma$ has the property that for some edge $e$ of $\Gamma$ and some $n\ge 1$ the path $v$ is a subpath of $f^n(e)$.  Kapovich and Lustig also proved in \cite{KL5} that $diag(L_{BFH}(\phi))$ is obtained from $L_{BFH}(\phi)$ by adding finitely many $F_N$ orbits of \lq\lq diagonal" leaves $(x,y)$ of a special kind. These extra \lq\lq diagonal leaves" play a similar role to the diagonals of ideal polygons given by complimentary regions for the lift to $\Hy^2$ of the stable geodesic lamination of a pseudo-anosov homeomorphism of a closed hyperbolic surface. 

In \cite{DKT} Dowdall, Kapovich and Taylor generalize the above description of $L_i$ to the case of word-hyperbolic extensions $E_\Gamma$ of $F_N$ determined by purely atoroidal ``convex cocompact" subgroups $\Gamma\le Out(F_N)$. See \cite[Corollary 5.3]{DKT} for details.


\begin{thebibliography}{ABC}


\bibitem{Ag}
I.~Agol, \emph{Tameness of hyperbolic 3-manifolds}, preprint, 2004; arXiv:0405568




\bibitem{ABT}
J. W. Anderson, P. Bonfert-Taylor, and E. Taylor, \emph{Convergence groups, Hausdorff dimension, and a theorem of Sullivan and Tukia.}  Geom. Dedicata, \textbf{103} (2004), 51--67


\bibitem{AHM}
B.\ Aebischer, S. \ Hong and D. McCullough, {\em Recurrent geodesics and controlled concentration points.}
Duke. Math. J., \textbf{75} (1994), no. 3, 759--774.

\bibitem{BR}
O.~Baker and T.~Riley, \emph{Cannon~Thurston maps do not always
  exist.} Forum of Mathematics, Sigma, \textbf{1} (2013), e3 (11 pages)
  
\bibitem{BRi}
O.~Baker and T.~Riley, \emph{Cannon-Thurston maps, subgroup distortion, and hyperbolic hydra},  preprint, 2012; arXiv:1209.0815
  

\bibitem{BeMa}
A. F. Beardon,  and B. Maskit,
\emph{Limit points of Kleinian groups and finite sided fundamental polyhedra.}
Acta Math., \textbf{132} (1974), 1--12

\bibitem{BF92} M.~Bestvina and M.~Feighn, \emph{A combination theorem for negatively curved groups.} J. Differential Geom., \textbf{35} (1992), no. 1, 85--101


\bibitem{BH92} M.~Bestvina, and M.~Handel, \emph{Train tracks and
automorphisms of free groups.} Ann.  of Math.  (2), \textbf{135}
(1992), no.  1, 1--51



\bibitem{BFH97}
M.~Bestvina, M.~Feighn, and M.~Handel, \emph{Laminations, trees, and
irreducible automorphisms of free groups.} Geom.  Funct.  Anal.,
\textbf{7} (1997), no.  2, 215--244

\bibitem{BFH00}
M.~Bestvina, M.~Feighn, and M.~Handel, \emph{The Tits alternative for
${\rm Out}(F\sb n)$.  I. Dynamics of exponentially-growing
automorphisms.} Ann.  of Math.  (2), \textbf{151} (2000), no.  2,
517--623



\bibitem{BF10}
M.~Bestvina, and M.~Feighn, \emph{A hyperbolic $Out(F_n)$--complex.} Groups Geom. Dyn., \textbf{4} (2010), no. 1, 31--58


\bibitem{BF11}
M.~Bestvina, and M.~Feighn, \emph{Hyperbolicity of the complex of free
  factors.}  Adv. Math., \textbf{256} (2014), 104--155

\bibitem{BR12}
M. Bestvina and P. Reynolds, \emph{The boundary of the complex of free factors.} Duke Math. J., \textbf{164} (2015), no. 11, 2213--2251


\bibitem{Bo}
O.~Bogopolski, \emph{Introduction to group theory.} EMS Textbooks in Mathematics. European Mathematical Society (EMS), Z\"urich, 2008

\bibitem{Bon}
F.\ Bonahon, \emph{Bouts des vari\'{e}t\'{e}s hyperboliques de dimension 3.} Ann.\ of Math., {\bf 124}, (1986) 71--158.


\bibitem{Bow98}
B. H. Bowditch, \emph{A topological characterisation of hyperbolic groups.}
J. Amer. Math. Soc., \textbf{11} (1998), no. 3, 643--667

\bibitem{Bow99}
B. H. Bowditch, \emph{Convergence groups and configuration spaces.} in:  "Geometric group theory down under (Canberra, 1996)", 23--54, de Gruyter, Berlin, 1999

\bibitem{Bow07}
B. H. Bowditch, \emph{The Cannon-Thurston map for punctured-surface groups.}  Math. Z., \textbf{255} (2007), no. 1, 35--76

\bibitem{Bow02}
B. H. Bowditch, \emph{Stacks of hyperbolic spaces and ends of 3 manifolds}, in: "Geometry and Topology Down Under'', Contemporary Mathematics Vol. 597, (eds. C.D.Hodgson, W.H.Jaco, M.G.Scharlemann, S.Tillman),  American Mathematical Society (2013), 65--138

\bibitem{BG}
M. R. Bridson, and D. Groves, \emph{The quadratic isoperimetric inequality for mapping tori of free group automorphisms.} Mem. Amer. Math. Soc., \textbf{203} (2010), no. 955

\bibitem{BridHae}
M.~R. Bridson and Andr{\'e} Haefliger. {\em Metric spaces of non-positive curvature}, volume 319 of {\em
  Grundlehren der Mathematischen Wissenschaften [Fundamental Principles of
  Mathematical Sciences]}.  Springer-Verlag, Berlin, 1999.

\bibitem{Br}
P. Brinkmann, \emph{Hyperbolic automorphisms of free groups.} Geom. Funct. Anal., \textbf{10} (2000), no. 5, 1071--1089



\bibitem{ELC2} J. F. Brock, R. D. Canary, and Y. N. Minsky. \newblock \emph{The classification of {K}leinian surface groups, {II}: {T}he ending  lamination conjecture.} \newblock {Ann. of Math. (2)}, \textbf{176} (2012), no.1, 1--149

\bibitem{CG}
D. Calegari and D. Gabai, \emph{Shrinkwrapping and the taming of hyperbolic 3-manifolds.} J. Amer. Math. Soc., \textbf{19} (2006), no. 2, 385--446

\bibitem{CT}
J.~W.~Cannon, and W.~P.~Thurston, 
\emph{Group invariant Peano curves.} 
Geom. Topol., \textbf{11} (2007), 1315--1355

\bibitem{CaJu}
A. Casson and D. Jungreis.
\newblock \emph{Convergence groups and {S}eifert fibered {$3$}-manifolds.}
\newblock {Invent. Math.}, \textbf{118} (1994), no. 3, 441--456

\bibitem{CB}
A.~J. Casson and S.~A. Bleiler.
\newblock {\em Automorphisms of surfaces after {N}ielsen and {T}hurston.}
  volume~9 of {\em London Mathematical Society Student Texts}.
\newblock Cambridge University Press, Cambridge, 1988.



\bibitem{CH12}
 T.~Coulbois, A.~Hilion, \emph{Botany of irreducible automorphisms of free groups.}  Pacific J. Math.,  \textbf{256} (2012),  no. 2, 291--307
 
 \bibitem{CH14}
 T.~Coulbois, A.~Hilion, 
 \emph{Rips induction: index of the dual lamination of an $\mathbb R$-tree.} Groups Geom. Dyn., \textbf{8} (2014), no. 1, 97--134
 

\bibitem{CHL}
 T.~Coulbois, A.~Hilion and M. Lustig,
 \emph{Non-unique ergodicity, observers' topology and the dual algebraic lamination for $\mathbb R$-trees.}
Illinois J. Math., \textbf{51} (2007), no. 3, 897--911

\bibitem{CHL1}
 T.~Coulbois, A.~Hilion and M. Lustig,
 \emph{$\mathbb R$-trees and
 laminations for free groups I: Algebraic laminations.}  J. Lond. Math. Soc. (2),  \textbf{78}  (2008),  no. 3, 723--736

\bibitem{CHL2}
T.~Coulbois, A.~Hilion and M. Lustig,
\emph{ $\mathbb R$-trees and
 laminations for free groups II: The dual lamination of an $\mathbb
 R$-tree.} J. Lond. Math. Soc. (2),  \textbf{78}  (2008),  no. 3, 737--754


\bibitem{DKT}
S. Dowdall, I. Kapovich and S. J. Taylor, \emph{Cannon-Thurston maps for hyperbolic free group extensions.} Israel J. Math., to appear; arXiv:1506.06974


\bibitem{Fen}
S. Fenley.
\newblock \emph{Ideal boundaries of pseudo-{A}nosov flows and uniform convergence
  groups with connections and applications to large scale geometry.}
\newblock {Geom. Topol.}, \textbf{16} (2012), no. 1, 1--110

\bibitem{F80}
W.~J. Floyd.
\newblock \emph{Group completions and limit sets of {K}leinian groups.}
\newblock {Invent. Math.}, \textbf{57} (1980), no. 3, 205--218

\bibitem{Fre}
E. Freden, \emph{Negatively curved groups have the convergence property. I.} Ann. Acad. Sci. Fenn. Ser. A I Math., \textbf{20} (1995), no. 2, 333--348

\bibitem{Gab92}
D. Gabai.
\newblock \emph{Convergence groups are {F}uchsian groups.}
\newblock {Ann. of Math. (2)}, \textbf{136} (1992), no. 3, 447--510


\bibitem{Ger}
V. Gerasimov, \emph{Expansive convergence groups are relatively hyperbolic.} Geom. Funct. Anal., \textbf{19} (2009), no. 1, 137--169

\bibitem{Ger12}
V. Gerasimov, \emph{Floyd maps for relatively hyperbolic groups.} Geom. Funct. Anal., \textbf{22} (2012), no. 5,1361--1399


\bibitem{GP09}
V. Gerasimov and L. Potyagailo, 
\emph{Quasi-isometric maps and Floyd boundaries of relatively hyperbolic groups.} J. Europ. Math. Soc. (JEMS), \textbf{15} (2013), no. 6, 2115--2137


\bibitem{GP13}
V. Gerasimov and L. Potyagailo, 
\emph{Similar relatively hyperbolic actions of a group}, preprint, 2013; arXiv:1305.6649


\bibitem{GeMa}
F. W. Gehring, and G. J. Martin, 
\emph{Discrete quasiconformal groups. I.} 
Proc. London Math. Soc. (3), \textbf{55} (1987), no. 2, 331--358








\bibitem{HM11}
M.~Handel and L.~Mosher, \emph{Lee Axes in outer space.} Mem. Amer. Math. Soc.,  \textbf{213} (2011), no. 1004



\bibitem{HM13}
M. Handel and L.~Mosher,
\emph{Subgroup decomposition in $Out(F_n)$: Introduction and Research
Announcement}, preprint, 2013; arXiv:1302.2681


\bibitem{JKOL}
W.\ Jeon, I.\ Kim and K.\ Ohshika with C.\ Lecuire,
\emph{Primitive stable representations of free Kleinian groups.} Israel J.\ Math., \textbf{199} (2014), no. 2, 841--866

\bibitem{JO}
W.\ Jeon, K. \ Ohshika, {\em Measurable rigidity for Kleinian groups.} Ergodic theory and Dynamical systems, to appear; published online June 2015; DOI: 10.1017/etds.2015.15


\bibitem{Ka99}
I. Kapovich, \emph{A non-quasiconvexity embedding theorem for hyperbolic groups.}   Math Proc. Cambridge Phil. Soc. \textbf{127} (1999), no. 3,  pp. 461--486

\bibitem{Ka}
I. Kapovich, \emph{Algorithmic detectability of iwip automorphisms.}
Bulletin of London Math. Soc.,  \textbf{46} (2014), no. 2, 279--290.



\bibitem{KB}
I. Kapovich, and N. Benakli, \emph{Boundaries of hyperbolic groups,} in:  "Combinatorial and geometric group theory (New York, 2000/Hoboken, NJ, 2001)", 39--93, Contemp. Math., 296, Amer. Math. Soc., 2002


\bibitem{KL2}
I.~Kapovich and M.~Lustig, \emph{Geometric Intersection Number and analogues of the Curve Complex for free
groups.} Geometry \& Topology, \textbf{13} (2009),  1805--1833


\bibitem{KL3}
I.~Kapovich and M.~Lustig, \emph{Intersection form, laminations and
currents on free groups.} Geom. Funct. Anal. (GAFA), \textbf{19} (2010), no. 5, pp. 1426--1467


\bibitem{KL5}
I. Kapovich and M. Lustig, \emph{Invariant laminations for irreducible
  automorphisms of free groups.}  Quarterly J. Math., \textbf{65} (2014), no. 4, 1241--1275


\bibitem{KL6}
I. Kapovich and M. Lustig,\emph{Cannon-Thurston fibers for iwip automorphisms of $F_N$.}  J. Lond. Math. Soc., \textbf{91} (2015), no. 1, 203--224

\bibitem{KM}
I.~Kapovich and A.~Myasnikov, \emph{Stallings foldings and the
subgroup structure of free groups.} J. Algebra, \textbf{248} (2002), no
2, 608--668



\bibitem{KS96}
I.~Kapovich, and H.~Short, \emph{Greenberg's theorem for quasiconvex subgroups of word hyperbolic groups.}  Canad. J. Math.,  \textbf{48}  (1996),  no. 6, 1224--1244

\bibitem{K95}
M. Kapovich, \emph{On the absence of Sullivan's cusp finiteness theorem in higher dimensions}, in: "Algebra and analysis (Irkutsk, 1989)", (L A Bokut', M Hazewinkel, Y G Reshetnyak, editors), Amer. Math. Soc. Transl. Ser. 2 163, Amer. Math. Soc. (1995) 


\bibitem{KK00}
M.~Kapovich, and B.~Kleiner, \emph{Hyperbolic groups with low-dimensional boundary.} Ann. Sci. Ecole Norm. Sup. (4) \textbf{33} (2000), no. 5, 647--669


\bibitem{KeLe1}
R. Kent, and C. J. Leininger,  \emph{Shadows of mapping class groups: capturing convex cocompactness.} Geom. Funct. Anal., \textbf{18} (2008), no. 4, 1270--1325

\bibitem{KeLe2}
R. Kent, and C. J. Leininger,  \emph{Uniform convergence in the mapping class group.} Ergodic Theory Dynam. Systems, \textbf{28} (2008), no. 4, 1177--1195

\bibitem{Kl99}
E. Klarreich,
\newblock \emph{Semiconjugacies between {K}leinian group actions on the {R}iemann
  sphere.}
\newblock {\em Amer. J. Math.}, \textbf{121} (1999), no. 5, 1031--1078


\bibitem{LLR}
C. Leininger, D. Long, and A. W. Reid, \emph{Commensurators of finitely generated nonfree Kleinian groups.}  Algebr. Geom. Topol.,  \textbf{11} (2011), no. 1, 605--624

\bibitem{LMS}
C. J. Leininger, M.  Mj,  and S. Schleimer, \emph{The universal Cannon-Thurston map and the boundary of the curve complex.} Comment. Math. Helv., \textbf{86} (2011), no. 4, 769--816.

\bibitem{LL}
G.~Levitt and M.~Lustig, \emph{Irreducible automorphisms of $F_n$ have
North-South dynamics on compactified outer space.} J. Inst.  Math.
Jussieu, \textbf{2} (2003), no.  1, 59--72


\bibitem{Mc01}
C.~T. McMullen.
\newblock \emph{Local connectivity, {K}leinian groups and geodesics on the blowup of the torus.}
\newblock {Invent. Math.}, \textbf{146} (2001), no. 1, 35--91

\bibitem{Min94}
Y. Minsky, 
\emph{On rigidity, limit sets, and end invariants of hyperbolic 3-manifolds.}
J. Amer. Math. Soc., \textbf{7} (1994), no. 3, 539--588

\bibitem{ELC1} Y. Minsky. \newblock {\em The classification of {K}leinian surface groups. {I}. {M}odels and  bounds.} \newblock {Ann. of Math. (2)}, \textbf{171} (2010), no. 1, 1--107

\bibitem{M1}
M.~Mitra, \emph{Ending laminations for hyperbolic group extensions.}  Geom. Funct. Anal.,  \textbf{7}  (1997),  no. 2, 379--402

\bibitem{M2}
M.~Mitra, \emph{Cannon-Thurston maps for hyperbolic group extensions.}
Topology,  \textbf{37}  (1998),  no. 3, 527--538

\bibitem{M4}
M.~Mitra, \emph{Cannon-Thurston maps for trees of hyperbolic metric spaces.}  J. Differential Geom.,  \textbf{48}  (1998),  no. 1, 135--164.

\bibitem{Miy}
H~Miyachi.
\newblock \emph{Semiconjugacies between actions of topologically tame {K}leinian
  groups.}
\newblock {2002, preprint}.



\bibitem{M07}
M. Mj, \emph{Ending laminations and Cannon-Thurston maps. With an appendix by Shubhabrata Das and Mj.} Geom. Funct. Anal., \textbf{24} (2014), no. 1, 297--321

\bibitem{M09}
M. Mj, \emph{Cannon-Thurston maps for pared manifolds of bounded geometry.} Geom. Topol., \textbf{13} (2009), no. 1, 189--245

\bibitem{M0910}
M. Mj.
\newblock \emph{Cannon-{T}hurston maps, i-bounded geometry and a theorem of
  {M}c{M}ullen.}
\newblock In {\em Actes du {S}\'eminaire de {T}h\'eorie {S}pectrale et
  {G}\'eometrie. {V}olume 28. {A}nn\'ee 2009--2010}, volume~28 of {\em S\'emin.
  Th\'eor. Spectr. G\'eom.}, pages 63--107. Univ. Grenoble I, Saint, 2010.

\bibitem{M10}
M. Mj, \emph{Cannon-Thurston maps and bounded geometry.} in: ``Teichm{\"u}ller theory and moduli problems'', 489--511, Ramanujan Math. Soc. Lect. Notes Ser., 10, Ramanujan Math. Soc., Mysore, 2010


\bibitem{M10b}
M. Mj, \emph{Cannon-Thurston Maps for Kleinian Groups},   arXiv:1002.0996

\bibitem{M11}
M. Mj, \emph{On discreteness of commensurators.} Geom. Topol., \textbf{15} (2011), no. 1, 331--350

\bibitem{M14}
M. Mj,
\newblock \emph{Cannon--{T}hurston maps for surface groups.} Ann. of Math. (2), \textbf{1979} (2014), no. 1, 1--80


\bibitem{MP11}
M. Mj, and A. Pal, \emph{Relative hyperbolicity, trees of spaces and Cannon-Thurston maps.} Geom. Dedicata, \textbf{151} (2011), 59--78



\bibitem{Q}
M.~Queff{\'e}lec.
\newblock {\em Substitution dynamical systems---spectral analysis}, volume 1294
  of {\em Lecture Notes in Mathematics}.
\newblock Springer-Verlag, Berlin, second edition, 2010



\bibitem{Rivin}
I.~Rivin,
\emph{Zariski density and genericity.}
Int. Math. Res. Not. IMRN, (2010), no. 19, 3649--3657


\bibitem{Sou}
J. Souto,
\newblock {\em Cannon-{T}hurston maps for thick free groups.}
\newblock 2006, preprint,\\ 
  http://www.math.ubc.ca/$\sim$jsouto/papers/Cannon-Thurston.pdf.


\bibitem{Sul82}
D. \ Sullivan, {\em Discrete conformal groups and measurable
  dynamics.} Bull. Amer. Math. Soc., \textbf{6} (1982), 57--73.


\bibitem{Swe}
E. Swenson,  \emph{Quasi-convex groups of isometries of negatively curved spaces.} Topology Appl., \textbf{110} (2001), no. 1, 119--129

\bibitem{Tu89}
P.\ Tukia, {\em A rigidity theorem for Mobius groups.} Invent. Math., \textbf{97} (1989), 405--431


\bibitem{Tu94}
P. Tukia, \emph{Convergence groups and Gromov's metric hyperbolic spaces.} New Zealand J. Math., \textbf{23} (1994), no. 2, 157--187

\bibitem{Tu98}
P. Tukia, 
\emph{Conical limit points and uniform convergence groups.}
J. Reine Angew. Math., \textbf{501} (1998), 71--98


\bibitem{Th}
W.~P. Thurston.
\newblock {\em {The geometry and topology of three-manifolds}}.
\newblock Lecture Notes from Princeton University, {1978--1980}.


\bibitem{Y}
A. Yaman, \emph{A topological characterisation of relatively hyperbolic groups.} J. Reine Angew. Math., \textbf{566} (2004), 41--89

\end{thebibliography}
\end{document}